\documentclass{amsart}
\usepackage{MnSymbol}
\usepackage[mathscr]{euscript}

\usepackage{hyperref}
\hypersetup{
    colorlinks,
    citecolor=black,
    filecolor=black,
    linkcolor=black,
    urlcolor=black
}

\usepackage{comment}
\usepackage{stmaryrd} 

\DeclareFontFamily{OT1}{pzc}{}
\DeclareFontShape{OT1}{pzc}{m}{it}{<-> s * [1.100] pzcmi7t}{}
\DeclareMathAlphabet{\mathpzc}{OT1}{pzc}{m}{it}

    \usepackage{etoolbox}
    \patchcmd{\section}{\scshape}{\large\bfseries}{}{}
    \makeatletter
    \renewcommand{\@secnumfont}{\bfseries}
    \makeatother

\usepackage{tikz-cd}
\tikzcdset{arrow style=math font}
\tikzcdset{scale cd/.style={every label/.append style={scale=#1},
    cells={nodes={scale=#1}}}}

\usepackage[maxbibnames=99]{biblatex}
\usepackage{booktabs}
\bibliography{references}

\numberwithin{equation}{section}
\newtheorem{theorem}{Theorem}[section]
\newtheorem*{theorem*}{Theorem}
\newtheorem*{conjecture*}{Conjecture}
\newtheorem{corollary}[theorem]{Corollary}
\newtheorem{lemma}[theorem]{Lemma}
\newtheorem{proposition}[theorem]{Proposition}
\theoremstyle{definition}

\newtheorem{remark}[theorem]{Remark}
\newtheorem{example}[theorem]{Example}

\def\NN{\mathsf{N}}

\def\PPP{\PPP}

\def\Ker{\mathsf{Ker}}

\def\epi{\twoheadrightarrow}

\def\mono{\rightarrowtail}
\def\ZZ{\mathbb{Z}}
\def\Pre{\mathsf{Pre}}
\def\KK{\mathbb{K}}
\def\MC{\mathsf{MC}}
\def\MH{\mathsf{MH}}
\def\PH{\mathsf{PH}}
\def\CCC{\mathpzc{C}}
\def\RRR{\mathpzc{R}}
\def\SS{\mathscr{N}}
\def\MM{\mathscr{M}}
\def\Cay{\mathsf{Cay}}
\def\Rel{\mathsf{Rel}}
\def\QQ{\mathbb{Q}}
\def\RR{\mathcal{R}}
\def\AA{\mathcal{A}}
\def\dist{\mathsf{dist}}

\let\oldtocsection=\tocsection 
\let\oldtocsubsection=\tocsubsection 
\renewcommand{\tocsection}[2]{\hspace{0mm}\oldtocsection{#1}{#2}}
\renewcommand{\tocsubsection}[2]{\hspace{1em}\oldtocsubsection{#1}{#2}}

\title{On the path homology of Cayley digraphs and covering digraphs}

\author{Shaobo Di}
\address{
Yanqi Lake Beijing Institute of Mathematical Sciences and Applications (BIMSA)}
\email{dishaobo@bimsa.cn}

\author{Sergei O. Ivanov}
\address{
Yanqi Lake Beijing Institute of Mathematical Sciences and Applications (BIMSA)}
\email{ivanov.s.o.1986@gmail.com, ivanov.s.o.1986@bimsa.cn}

\author{Lev Mukoseev}
\address{
St. Petersburg University}
\email{la.mukoseev@gmail.com}

\author{Mengmeng Zhang}
\address{
Yanqi Lake Beijing Institute of Mathematical Sciences and Applications (BIMSA)}
\email{mengmengzhang@bimsa.cn}

\thanks{The work is supported by  Beijing Institute of Mathematical Sciences and Applications}

\begin{document}

\maketitle

\begin{abstract}
We develop a theory of covering digraphs, similar to the theory of covering spaces. By applying this theory to Cayley digraphs, we build a ``bridge'' between GLMY-theory and group homology theory, which helps to reduce path homology calculations to group homology computations. We show some cases where this approach allows us to fully express path homology in terms of group homology. To illustrate this method, we provide a path homology computation for the Cayley digraph of the additive group of rational numbers with a generating set consisting of inverses to factorials. The main tool in our work is a filtered simplicial set associated with a digraph, which we call the filtered nerve of a digraph.
\end{abstract}

\section*{Introduction}

Grigoryan, Lin, Muranov, and Yau developed the theory of path homology for digraphs in an unpublished preprint \cite{grigor2012homologies}, which is now referred to as GLMY-theory or GLMY-homology theory \cite{grigor2014homotopy, grigor2017homologies, grigor2020path, grigor2014graphs, grigor2018path,grigor2018path2,grigor2022advances,ivanov2022simplicial}. This theory is in many respects parallel to the theory of homology of spaces. In particular, it is homotopy invariant in a sense appropriate for digraphs and there are K\"unneth formulas for box product and join. They also defined a notion of the fundamental group of a digraph, whose abelianization coincides with the first integral path homology group \cite{grigor2014homotopy}. The notion of fundamental group was later extended to the notion of fundamental groupoid by Grigor’yan, Jimenez, and Muranov  \cite{grigor2018fundamental}.  

In our work, we continue the analogy between algebraic topology for spaces and ``algebraic topology for digraphs'' by developing a theory of coverings for digraphs. To be more precise, we develop a theory of $l$-fundamental groups and $l$-coverings of digraphs,  where $l\geq 1$ is a natural parameter. For $l=1$ this theory is just the ordinary theory of fundamental groups and covering spaces, where a digraph is treated as a one-dimensional space. The most important case for us is $l=2,$ because in this case this theory is compatible with GLMY-theory, and the $2$-fundamental group $\pi_1^2(X)$ coincides with the group considered in \cite{grigor2014homotopy,grigor2018fundamental}. A digraph is called connected if its underlying undirected graph is connected. We prove that for the universal $2$-covering $U\to X$ of a connected digraph $X$ there is a spectral sequence
\[H_*( \pi^2_1(X),\PH_*(U) ) \Rightarrow \PH_*(X),\]
where $\PH_*(-)$ denotes the path homology, and $H_*(\pi^2_1(X),-)$ the group homology of the group $\pi_1^2(X)$. This spectral sequence is similar to the Serre spectral sequence of the homotopy fibration sequence $\mathpzc{U}\to \mathpzc{X} \to B\pi_1(\mathpzc{X})$ for a space $\mathpzc{X}$ and its universal cover $\mathpzc{U}.$ This spectral sequence is especially useful when $\PH_n(U)=0$ for $n\geq 1,$ because in this case we have an equation 
\[\PH_*(X)=H_*(\pi_1^2(X)),\] 
which reduces the computation of the path homology to the homology of the fundamental group. 

Next, we apply the obtained theory to the case of Cayley digraphs. We describe generating subsets $S$ of a group $G$ whose Cayley digraphs have a trivial $l$-fundamental group. We call them $l$-freely generating subsets. 
Roughly speaking, these are subsets of groups such that all relations between its elements are determined by some relations of the form $s_1 \dots s_n = t_1 \dots t_m$ for $n, m \leq l, s_i,t_j\in S.$ We show that the universal $2$-covering of the Cayley digraph $\Cay(G,S)$ is a Cayley digraph of a $2$-freely generating subset $\tilde{S}$ of a group $\tilde{G}$, which is obtained from $G$ by ``removing'' all relations that are not of the form $s_1 \dots s_n = t_1 \dots t_m$ for $0 \leq n, m \leq 2$. Furthermore, we obtain that there is a short exact sequence
\[1\longrightarrow \pi^2_1(\Cay(G,S))\longrightarrow \tilde G \longrightarrow G \longrightarrow 1,\]
and, if we set $\rho:=\Ker(\tilde G \to G),$ we obtain a spectral sequence
\[ 
H_*( \rho, \PH_*( \Cay(\tilde G,\tilde S) )) 
\Rightarrow  
\PH_*( \Cay(G,S)). 
\]
In a sense, this spectral sequence reduces the computation of path homology of Cayley digraphs to the computation of path homology of Cayley digraphs of $2$-freely generating subsets of groups and the computation of group homology with coefficients in certain modules. As well as in the general case, this spectral sequence is especially useful when $\Cay(\tilde G,\tilde S)$ has trivial reduced path homology, because in this case we have an isomorphism
\[\PH_*(\Cay(G,S))\cong H_*(\rho).\]
Our experience with concrete calculations indicates that this happens quite often but not always. For further investigation of path homology of Cayley digraphs, it would be important to understand in which cases reduced path homology of a Cayley digraph of a $2$-freely generating subset of a group is trivial.

The most convenient class of groups for applying the theory we have developed is the class of abelian groups.  It is well known that for a torsion free abelian group $\rho$ its homology with coefficients in a commutative ring $\KK$ can be described as the exterior algebra $H_*(\rho) = \Lambda^*(\rho\otimes_\ZZ \KK).$ Using this equality, we prove a general theorem, which allows us to compute the path homology of Cayley digraphs for many subsets of abelian groups (Theorem \ref{th:abelian}).
For example, we compute path homology of the Cayley digraph of $\QQ$ with the generating subset consisting of the inverses of factorials:
\[ \PH_*( \Cay( \QQ , \{ 1/n!\mid n\geq 2
\} )) \cong  \Lambda^*( \KK^{\oplus \mathbb{N}} ).\]
Thereby, path homology modules of this Cayley digraph are infinitely generated free $\KK$-modules in all positive degrees.

Before developing the theory of $l$-fundamental groups and $l$-coverings of digraphs, we needed to formulate the GLMY-theory (and magnitude homology theory) in the style of Asao \cite{asao2023magnitude,asao2023magnitude2}. Following Asao and a remark of Hepworth-Willerton \cite[Remark 46]{hepworth2017categorifying} we construct a filtered space associated with a digraph $X$ 
\[\SS^0(X)\subseteq \SS^1(X) \subseteq \SS^2(X) \subseteq \dots \subseteq \SS(X),\]
that we call the filtered nerve of the digraph. Here for each $l\geq 0$ the simplicial set $\SS^l(X)$ is in some sense the minimal simplicial set generated by pathes of length $\leq l.$ In particular, the dimension of $\SS^l(X)$ is at most $l.$ Asao showed that the spectral sequence of this filtered space contains the information about both magnitude homology and path homology. Therefore, we refer to it as the magnitude-path spectral sequence. In this paper, we adhere to the ideology that the whole ``homotopy digraph theory'' should be interpreted in terms of the filtered space $\SS^*(X)$. In particular, we define the $l$-fundamental group of a digraph as the fundamental group of $\SS^l(X)$, and the $l$-coverings of the digraph correspond to the coverings of the space $\SS^l(X).$

\tableofcontents

\section{\bf Filtered nerve of a digraph}

\subsection{Paths and distance in a digraph}

A digraph is a couple $X=(X_0,A(X)),$ 
where $X_0$ is a set and $A(X)\subseteq (X_0\times X_0)\setminus \Delta(X_0)$
and $\Delta(X_0)=\{(x,x) \mid x\in X_0\}.$ Elements of $X_0$ are called vertices and elements of $A(X)$ are called arrows (or directed edges, or arcs). We set
\begin{equation}
X_1:=A(X)\cup \Delta(X_0).    
\end{equation}
A morphism of digraphs $f:X\to Y$ is a map $f:X_0\to Y_0$ such that $(v,u)\in  X_1$ implies $(f(v),f(u))\in  Y_1.$ A disadvantage of the set of arrows $A(X)$ is that it is not functorial by $X,$ so we prefer to use $X_1.$

An $n$-path in digraph $X$ from a vertex $x$ to a vertex $x'$ is a sequence of vertices $(x_0,\dots,x_n)$ such that $(x_i,x_{i+1})\in X_1$ and $x=x_0,x'=x_n.$ An $n$-path is called regular, if $x_i\ne x_{i+1}.$ For $x,x'\in X_0$ we denote by $\dist(x,x')$ the minimal $n$ such that there is an $n$-path from $x$ to $x'.$ If such a path does not exists, we set $\dist(x,x')=+\infty.$ This gives a structure of extended quasi-metric space on $X_0$. In particular, for any vertex $x$ we have $\dist(x,x)=0,$ and for any three vertices $x,y,z$ we have the triangle inequality
\begin{equation}\label{eq:Triangle_inequality}
\dist(x,z) \leq  \dist(x,y) + \dist(y,z).
\end{equation}
Note that a map $f:X_0\to Y_0$ is a morphism of digraphs $X\to Y$ if and only if for any vertices $x,x'$ of $X$ we have 
\begin{equation}\label{eq:dist_of_images}
\dist(f(x),f(x'))\leq \dist(x,x').   
\end{equation}
We find it convenient to write the fact that a pair $(x,x')$ is an arrow in the form of the equality $\dist(x,x')=1,$ and the fact that there is a path from $x$ to $x'$ in the form of the inequality $\dist(x,x')<\infty.$

\subsection{Accessible sequences and the nerve of a digraph}

We say that a sequence $(x_0,\dots,x_n)$ of vertices of $X$ is \emph{accessible}, if $\dist(x_i,x_{i+1})<\infty$ for each $0\leq i<n.$ We consider a simplicial set $\SS(X)$ whose $n$-simplices are accessible sequences  $(x_0,\dots,x_n).$ The face and degeneracy maps are defined by the formulas 
\begin{equation}\label{eq:d_i-formulas}
\begin{split}
d_i(x_0,\dots,x_n) &= (x_0,\dots, \hat x_i,\dots,x_n), \\
s_i(x_0,\dots,x_n) &= (x_0,\dots,x_i,x_i,\dots,x_n).
\end{split}
\end{equation}
The simplicial set $\SS(X)$ will be called \emph{nerve} of $X.$ Degenerate simplices of this simplicial set are called non-regular accessible sequences, and non-degenerate simplices are called regular accessible sequences.

A preorder is a category whose hom-sets have cardinality not greater than $1.$ For any digraph $X$ we construct a preorder $\Pre(X)$ whose objects are vertices of $X$ and the hom-sets are described as follows
\begin{equation}
\Pre(X)(x,y) = \begin{cases} 
\{(x,y)\}, & \dist(x,y)<\infty \\ 
\emptyset, & \text{ else.}
\end{cases}
\end{equation}

\begin{proposition} \label{prop:nrv}
The nerve of $X$ is naturally isomorphic to the nerve of the preorder
\[
\SS(X) \cong {\sf Nrv}(\Pre(X)).
\]  
\end{proposition}
\begin{proof}
The correspondence $(x_0,\dots,x_n)\mapsto ((x_0,x_1),(x_1,x_2),\dots,(x_{n-1},x_n))$ defines the isomorphism. 
\end{proof}

\subsection{Filtration of the nerve} 
Following Asao \cite[Def.2.3]{asao2023magnitude} we define the length of an accessible sequence by the formula 
\begin{equation}
{\sf L}((x_0,\dots,x_n)) := \sum_{i=0}^{n-1} \dist( x_i,x_{i+1}).
\end{equation}
Using the triangle inequality  \eqref{eq:Triangle_inequality}, it is easy to check that  the following properties are satisfied
\begin{equation}\label{eq:L}
{\sf L}(d_i(\sigma)) \leq {\sf L}(\sigma), \hspace{1cm} {\sf L}(s_i(\sigma)) = (\sigma) 
\end{equation}
for any $\sigma\in \SS(X)_n$ and any $0\leq i\leq n.$

For any natural $l$ we denote by 
$\SS^l(X)_n$ the subset of $\SS(X)_n$ consisting of accessible sequences $\sigma$ such that ${\sf L}(\sigma)\leq l.$ The inequalities \eqref{eq:L} show that these subsets define a simplicial subset $\SS^l(X) \subseteq \SS(X).$ For all $l=0,1,\dots$ it gives an exhaustive filtration of $\SS(X)$
\begin{equation}
\SS^0(X) \subseteq \SS^1(X) \subseteq \SS^2(X) \subseteq \dots \subseteq \SS(X), \hspace{5mm}
\SS(X) = \bigcup_{l=0}^\infty \SS^l(X).
\end{equation}
We also consider the quotients 
\begin{equation}
\MM^l(X):=\SS^l(X)/\SS^{l-1}(X)    
\end{equation}
for $l\geq 1$ and $\MM^0(X):= \SS^0(X) \sqcup *.$ The simplicial set $\MM^l(X)$ will be called $l$-th magnitude space of $X$. 

It is easy to see that any accessible sequence $(x_0,\dots, x_n)$ of length $<n$ is non-regular. Therefore all simplices of $\SS^l(X)_n$ are degenerate for $n>l.$ It follows that
\begin{equation}\label{eq:ner^l-dim}
    {\sf dim}(\SS^l(X)) \leq l.
\end{equation}
It is easy to see that $\SS^l(X)$ is the minimal simplicial subset of $\SS(X)$ containing all regular paths of length $\leq l.$

\subsection{Filtered nerve of the box  product}

The box product $X \square Y$ of digraphs $X$ and $Y$ is a digraph such that $(X\square Y)_0=X_0\times Y_0$  and $((x,y),(x',y'))\in (X \square Y)_1$ if and only if $((x,x'),(y,y'))\in (X_1 \times \Delta(Y_0))\cup (\Delta(X_0)\times Y_1)$
\begin{equation}
(X\square Y)_1 \cong (X_1 \times \Delta(Y_0))\cup (\Delta(X_0)\times Y_1).
\end{equation}

\begin{lemma} 
A sequence of couples of vertices 
$((x_0,y_0),\dots,(x_n,y_n))$ is a path in $X\square Y$ if and only if $(x_0,\dots,x_n)$ and $(y_0,\dots,y_n)$ are paths in $X$ and $Y$ and there exists a partition $I\sqcup J = \{1,\dots,n\}$ such that $x_{i-1}=x_i$ and $y_{j-1}=y_{j}$ for any $i\in I,j\in J.$
\end{lemma}
\begin{proof}
Obvious.
\end{proof}

\begin{corollary}\label{corollary:dist_box}
The distance in the box  product satisfy the following equation
\begin{equation}
\dist_{X\Box Y}((x,y),(x',y')) = \dist_{X}(x,x')+\dist_Y(y,y').
\end{equation}
\end{corollary}

\begin{proposition}[{cf.  \cite[Prop.41.]{hepworth2017categorifying}}]
For any digraphs $X,Y$ there is an natural isomorphism
$\SS(X \square Y)\cong \SS(X) \times \SS(Y),$
and for any $l$ this isomorphism induces a isomorphisms
\begin{equation*}
\SS^l(X\square Y) \cong   \bigcup_{l_1+l_2=l} \SS^{l_1}(X) \times  \SS^{l_2}(Y),
\hspace{5mm}
\MM^l(X\square Y) \cong \bigvee_{l_1+l_2=l} \MM^{l_1}(X) \wedge \MM^{l_2}(Y) 
\end{equation*}
\end{proposition}
\begin{proof} We define the map $f: \SS(X\square Y) \to \SS(X) \times \SS(Y)$ by the formula 
\[((x_0,y_0),\dots,(x_n,y_n))\mapsto ((x_0,\dots, x_n),(y_0,\dots,y_n)).\]
Corollary \ref{corollary:dist_box} implies that 
\begin{equation}
{\sf L}((x_0,y_0),\dots,(x_n,y_n)) = {\sf L} (x_0,\dots,x_n) + {\sf L}(y_0,\dots,y_n).
\end{equation}
The first two isomorphisms follow. For simplicity we set $Z^l:=\SS^l(Z).$ Then the last isomorphism follows from the formulas 
\begin{equation*}
(X^{l_1}\times Y^{l_2}) \cap (  X^{l_1'}\times Y^{l_2'}) = X^{{\sf min}(l_1,l_1')} \times Y^{{\sf min}(l_2,l'_2)} \subseteq f((X\square Y)^{l-1})
\end{equation*}
if $(l_1,l_2)\ne (l_1',l_2')$ and $l_1+l_2=l=l'_1+l'_2,$ and
\[ (X^{l_1}\times Y^{l_2}) \cap f(
 (X\square Y )^{l-1} ) = (
X^{l_1-1} \times 
Y^{l_2}) \cup (
X^{l_1} \times 
Y^{l_2-1}).\]
\end{proof}

\subsection{Reminder on homology of simplicial sets}

Let $\KK$ be a commutative ring. For a simplicial $\KK$-module $M$ we denote by $C(M)$ the chain complex with components $C(M)_n=M_n$ and the differential $\partial_n:C(M)_n\to C(M)_{n-1}$ defined by $\partial_n=\sum_{i} (-1)^i d_i.$ We denote by $D(M)$ the subcomplex of degenerate simplices of $C(M)$ whose components are $D(M)_n=\sum_{i} s_i(M_{n-1}).$ The \emph{Moore complex} (or normalised complex) of $M$ is the quotient complex
$N(M) = C(M)/D(M).$
It is well known that  $N$ is an exact functor.
It is also well known that $D(M)$ is contractible and 
$H_*(N(M))\cong H_*(C(M)).$

For a simplicial set $S$ we denote by  $\KK(S)$ the free simplicial $\KK$-module whose module of $n$-simplices $\KK(S)_n=\KK(S_n)$ is freely generated by $S_n.$ Then the normalised complex of $S$ and homology of $S$ are defined by 
\begin{equation}
N(S) = N(\KK(S)), \hspace{1cm} H_*(S) = H_*(N(S)), 
\end{equation}
we will also use notation $C(S)=C(\KK(S))$ and $D(S)=D(\KK(S)).$
If $S'\subseteq S$ is a simplicial subset, then $\KK(S')$ is a simplicial submodule of $\KK(S)$, and, since $N$ is an exact functor, $N(S') \to N(S)$ is a monomorphism. We will usually identify it with its image $N(S')\subseteq N(S).$
If $S=(S,*)$ is a pointed simplicial set, the reduced versions of the normalized complex of $S$ and homology are defined by the formulas
\begin{equation}
\tilde N(S) = N(S)/N(*), \hspace{1cm} \tilde H_*(S) := H_*( \tilde N(S) ).
\end{equation}
For any simplicial subset $S'\subseteq S$ we define $S/S'$ as the pushout of the diagram $* \leftarrow S' \to S.$ If $S'$ is non-empty, then $(S/S')_n$ is just the quotient of $S_n$ by the equivalence relation, where all elements of $S'_n$ are identified. If $S'=\emptyset,$ then $S/\emptyset = S\sqcup *.$ Anyway  $S/S'$ has a natural base-point and 
\begin{equation}\label{eq:homology_of_quotient}
\tilde N(S/S') \cong N(S)/N(S'), \hspace{1cm} \tilde H_*(S/S') \cong H_*(N(S)/N(S')).
\end{equation}
If $S'$ is a pointed simplicial subset of $S$ we also have a formula $\tilde N(S/S')\cong \tilde N(S)/\tilde N(S'),$ which implies that there is a long exact sequence 
\begin{equation}\label{eq:connecting_homo}
\dots \to \tilde H_n(S') \to \tilde H_n(S) \to \tilde H_n(S/S') \overset{\delta}\to \tilde H_{n-1}(S') \to \dots.
\end{equation}
The homomorphism $\delta$ is called connecting homomorphism for the cofiber sequence $S'\hookrightarrow S \to S/S'$.

\subsection{Chain complex of regular accessible sequences}

Let $\KK$ be a commutative ring.  For a digraph $X$ we denote by $\Lambda(X)$ a chain complex such that 
$\Lambda_n(X)$  is an $\KK$-module 
spanned by \emph{accessible} sequences  $(x_0,\dots,x_n).$ The differential 
$\partial_n : \Lambda_n(X) \to \Lambda_{n-1}(X)$
is defined by the formula
\begin{equation}\label{eq:differential}
\partial_n(x_0,\dots,x_n)=\sum_{i=0}^n (-1)^i (x_0,\dots,\hat{x_i},  \dots,x_n).
\end{equation}
We denote by $D(X)$ a chain subcomplex of $\Lambda(X)$ generated by non-regular sequences. Finally we set
\begin{equation}
\RR(X)=\Lambda(X)/D(X).
\end{equation}
For an element $a\in \Lambda_n(X)$ we denote by $[a]$ its image in $\RR_n(X),$ and for an accessible sequence of vertices $(x_0,\dots,x_n)$ we set $[x_0,\dots,x_n]:=[(x_0,\dots,x_n)].$

\begin{proposition} For a digraph $X$ we have 
\begin{equation}
\Lambda(X) = C(\SS(X)), \hspace{1cm} \RR(X) = N(\SS(X)).
\end{equation}
Moreover, an accessible sequence of vertices $(x_0,\dots,x_n)$ is non-regular if and only if it is in $D(\SS(X)).$
\end{proposition}
\begin{proof}
Obvious. 
\end{proof}

\subsection{Magnitude homology}
For a digraph $X$ we denote by $\MC^{\leq l}(X)$ a chain subcomplex of $\RR(X)$ generated by the images of accessible sequences $[x_0,\dots,x_n]$ with ${\sf L}(x_0,\dots,x_n)\leq l.$ We also assume that $\MC^{\leq -1}(X)=0.$ We define the magnitude complex and magnitude homology by the formulas  
\begin{equation}
\MC^l(X) = \MC^{\leq l}(X)/\MC^{\leq l-1}(X), \hspace{1cm} \MH^l_n(X) = H_n(\MC^l(X)).
\end{equation}
We call $\MC^l(X)$ the $l$-th magnitude chain complex of $X.$ For an accessible sequence $\sigma=(x_0,\dots,x_n)$ such that ${\sf L}(\sigma)\leq l$ we denote by $\llbracket \sigma\rrbracket$ its image in $\MC^l(X).$ Then 
\begin{equation}\label{eq:basis2}
\{ \llbracket \sigma \rrbracket \mid {\sf L}(\sigma)=l, \sigma \text{ is regular} \}
\end{equation}
is a basis of $\MC^l(X).$ 
Since any regular sequence $(x_0,\dots, x_n)$ has length at least $n,$ we obtain 
\begin{equation}\label{eq:magnitude_under_diagonal}
\MC^{\leq l}_n(X)=\MC^l_n(X)=\MH^l_n(X)=0 \text{ for } n>l.
\end{equation}
We also obtain that
\begin{equation}\label{eq:MH^l_l}
\MH^n_n(X) = {\sf Ker}( \partial : \MC^n_n(X) \to \MC^n_{n-1}(X)).
\end{equation}

\begin{proposition}
Our definition of $\MC^l(X)$ (up to a natural isomorphism) coincides with the definition given by Asao \cite[Def.2.13]{asao2023magnitude}. 
\end{proposition}
\begin{proof} Since the digraph $X$ is fixed, we will omit $X$ in the notation in the proof $\MC^l:=\MC^l(X), \RR_n:=\RR_n(X)$ and so on. Denote the complex defined by Asao by $\overline{\MC}^l.$ 
Asao defines $\overline{\MC}^l_n$ as an $\KK$-submodule of $\RR_n$ spanned by  regular sequences $\sigma$ such that ${\sf L}(\sigma)=l.$ Then 
\begin{equation}\label{eq:basis1}
\{ [\sigma]\mid {\sf L}(\sigma)=l, \sigma \text{ is regular} \}    
\end{equation}
is a basis of $\overline{\MC}^l.$
Note that the assumption $ {\sf L}( \sigma )\leq l$ implies that the sequence $\sigma$ is accessible. Asao uses a map $(-)^l:\RR_n\to \RR_n$ such that for any accessible sequence $\sigma$ we have $\sigma^l=\sigma,$ if ${\sf L}(\sigma)=l,$ and $\sigma^l=0,$ if ${\sf L}( \sigma ) \ne l.$
He defines the differential $\bar \partial_n : \overline{\MC}^l_n \to \overline{\MC}^l_{n-1}$ by the formula $\bar \partial(\sigma)=\partial(\sigma)^l,$ where $\partial$ is the differential of $R.$

 Consider the map 
$\theta_n^l : \MC^{\leq l}_n \to \overline{\MC}^l_n $
defined as the restriction of $(-)^l.$ Note that $\Ker(\theta_n^l)=\MC^{\leq l-1}_n.$ We claim that $\theta^l=(\theta^l_n)_n$ is a morphism of complexes. Indeed, for any $a\in \MC^{\leq l}_n$ we have $a-a^l \in \MC^{\leq l-1}_n,$ and hence, $\theta^l_n(a)=\theta^l_n(a^l).$  Therefore $\bar \partial_n \theta_n^l(a) = \partial(a^l)^l=\partial(a)^l=\theta_{n-1}^l\partial(a).$ It follows that $\theta^l:\MC^{\leq l} \to \overline{\MC}^l$ is a morphism of complexes. Since $\Ker(\theta^l)=\MC^{\leq l-1},$ we obtain that $\theta^l$ induces a morphism 
$\tilde \theta^l : \MC^l\to \overline{\MC}^l,$ which induces a bijection on the bases \eqref{eq:basis1}, \eqref{eq:basis2}. Therefore $\tilde \theta^l$ is an isomorphism of complexes. 
\end{proof}

\begin{lemma}[{cf. \cite[Lemma 42]{hepworth2017categorifying}}]\label{lemma:MC-Moore} There are natural isomorphisms
\begin{equation}
\MC^{\leq l}(X) \cong  N(\SS^l(X)), \hspace{1cm} \MC^l(X) \cong  \tilde N(\MM^l(X)).
\end{equation}
\end{lemma}
\begin{proof}
The embedding $\SS^l(X) \hookrightarrow \SS(X) $ induces an embedding $N(\SS^l(X)) \hookrightarrow N(\SS(X))=\Lambda(X),$ whose image is obviously $\MC^{\leq l}(X).$ The first isomorphism follows. The second isomorphism follows from \eqref{eq:homology_of_quotient}.
\end{proof}

\begin{theorem}[{cf. \cite[Remark 45]{hepworth2017categorifying}}]
\label{the:magnitude-relative} The magnitude homology can be described as reduced homology of the magnitude space 
\begin{equation}
\MH^l_n(X)= \tilde H_n(\MM^l(X)).
\end{equation}
\end{theorem}
\begin{proof} It follows from 
Lemma \ref{lemma:MC-Moore} and Equation \eqref{eq:homology_of_quotient}. 
\end{proof}

\subsection{Path homology and GLMY-theory}
Denote by $\AA_n(X)$ a submodule of $\RR_n(X)$ generated by elements $[x_0,\dots,x_n],$ where $(x_0,\dots,x_n)$ is a path in $X.$ Then we denote by $\Omega(X)$ a maximal chain subcomplex of $\RR(X)$ such that its $n$-th component $\Omega_n(X)$ is a submodule in $\AA_n(X).$ Its components can be defined by the formula
$\Omega_n (X)=  \partial^{-1}(\AA_{n-1}(X))\cap \AA_n(X).$ Its homology is called path homology of $X$ or GLMY-homology of $X$
\begin{equation}
\PH_*(X)=H_*(\Omega(X)).
\end{equation}

\begin{proposition}
The chain complex $\Omega(X)$ is naturally isomorphic to the chain complex $\Omega^{\sf reg}(X)$ defined by Grigonyan, Lin, Muranov, Yau in \cite{grigor2012homologies}.
\end{proposition}
\begin{proof}
The definition of Grigonyan, Lin, Muranov, Yau differs from our in that they consider all sequences of vertices while we consider only accessible sequences. 
They define a complex that we denote by  $\Lambda'(X),$ whose components $\Lambda'(X)_n$ are freely generated by all sequences of vertices $(x_0,\dots,x_n)$ with the same formula for the differential \eqref{eq:differential}. This complex also has a structure of a simplicial module and they consider the quotient $\RR'(X)=N(\Lambda'(X)).$ The embedding $\Lambda(X)\subseteq \Lambda'(X)$ induces an embedding $\iota: \RR(X) \mono \RR'(X)$. They consider submodules $A'_n(X)\subseteq \RR'(X)_n$ generated by the images of paths of $X,$ and define $\Omega^{\sf reg}(X)$ as the maximal subcomplex of $R'(X)$ such that $\Omega^{\sf reg}_n(X) \subseteq \AA'_n(X).$ It is easy to see that $\iota(\AA_n(X))=\AA'_n(X).$ Then the assertion follows from \cite[Remark 3.1]{ivanov2022simplicial}.
\end{proof}

\begin{remark}
We prefer to use the term ``path'' in the ordinary combinatorial meaning. We use the term ``GLMY-chains'' or ``path chains'' for homogeneous elements of $\Omega(X),$ and ``GLMY-cycles'' or ``path cycles'' for the cycles of this complex.  
\end{remark}

Asao proved that $\Omega_n(X)$ is isomorphic to $\MH^n_n(X).$ Here we provide a more homotopical version of Asao's result. 

\begin{proposition}[{sf. \cite[Lemma 6.8]{asao2023magnitude}}]\label{th:asao}
The chain complex $\Omega(X)$ is isomorphic to the chain complex
\begin{equation}\label{eq:complex_GLMY-hom-th}
 \dots \overset{\delta_3}\longrightarrow  \tilde H_2(\MM^2(X)) \overset{\delta_2}\longrightarrow \tilde H_1( \MM^1(X)) \overset{\delta_1}\longrightarrow  \tilde H_0( \MM^0(X)),   
\end{equation}
where $\delta_n:\tilde H_n(\MM^n(X)) \to \tilde  H_{n-1}(\MM^{n-1}(X))$ is the connecting homomorphism for the cofiber sequence $\MM^{n-1}(X)\hookrightarrow \SS^n(X)/\SS^{n-2}(X) \to \MM^n(X).$  
\end{proposition}
\begin{proof} For simplicity in the proof we omit $X$ in notations $\MM^n:=\MM^n(X)$, $\AA_n:=\AA_n(X)$ \dots.
Since ${\sf dim}(\MM^n)\leq n,$ we have  $\tilde N(\MM^n)_{n+1}=0$ and $\tilde N(\MM^n)_n \cong  N(\SS^n)_n=\AA_n.$ Therefore, 
\begin{equation}
\tilde H_n(\MM^n) = \Ker( \AA_n \to N(\SS^n)_{n-1} / \AA_{n-1}).
\end{equation}
It follows that $\tilde H_n(\MM^n)\cong \partial^{-1}(\AA_{n-1})\cap \AA_n=\Omega_n.$  The map $\delta_n$ is the connecting homomorphism for the short exact sequence of complexes
\begin{equation}
  0 \to \tilde N(\MM^{n-1}) \to \tilde N(\SS^n/\SS^{n-2}) \to \tilde N(\MM^n) \to 0,
\end{equation}
which is induced by the differential of $\tilde N(\SS^n/ \SS^{n-2}).$ The differential on $\tilde N(\SS^n/\SS^{n-2})$ $ = N(\SS^n)/N( \SS^{n-2})$ coincides with the differential induced by the differential on $\Lambda.$ So $\delta_n$ is induced by the differential on $\Lambda$ and the differential of $\Omega$ is also induced by the differential on $\Lambda.$ Using this one can check that $\Omega$ is isomorphic to the complex \eqref{eq:complex_GLMY-hom-th}. 
\end{proof}

The chain complex $\RR(X)$ has a structure of a semi-simplicial module defined by the maps 
\begin{equation}
d_i : \RR(X)_n \longrightarrow \RR(X)_{n-1}, \hspace{5mm} [x_0,\dots,x_n] \mapsto [x_0,\dots, \hat x_i,\dots,x_n].
\end{equation}
This structure can be used to give an alternative description of $\Omega(X)_n.$

\begin{proposition}\label{prop:description_of_PC} The following holds
\[\Omega_n(X)=\{ a \in \AA_n(X) \mid d_i(a)\in  \AA_{n-1}(X) \text{ for any } 1\leq i\leq n-1 \}.\]
\end{proposition}
\begin{proof} For simplicity, in the proof we will omit $X$ in the notation: $R=\RR(X), \Lambda=\Lambda(X), D=D(X)$ and so on.
For a regular  accessible sequence $(x_0,\dots,x_n)$ we define its type as the sequence of the distances
$(\dist(x_0,x_1),\dist(x_1,x_2),\dots, \dist( x_{n-1},x_n)).
$ For a sequence of positive natural numbers $(k_1,\dots,k_n)$ we denote by $\RR_{(k_1,\dots,k_n)}$ the submodule of $\RR_n$ freely generated by accessible sequences of type $(k_1,\dots,k_n).$ Then 
\begin{equation}
\RR_n = \bigoplus_{(k_1,\dots,k_n)\in \mathbf{N}_{>0}^n} \RR_{(k_1,\dots,k_n)},
\end{equation}
and 
$\AA_n = \RR_{\mathbf{1}},$ where $\mathbf{1}=(1,\dots,1).$ For any element $r\in \RR_n$ we denote by $r_{(k_1,\dots,k_n)}$ its component in $\RR_{(k_1,\dots,k_n)}.$ Then $r\in A_n$ if and only if $r =r_{\mathbf{1}}.$  

For $1\leq i\leq n-1$ we denote by $t_i$ the vector $(1,\dots,1,2,1,\dots,1) \in \mathbb{N}^{n-1}_{>0}$ where $2$ is on the $i$-th position. Then we have $d_i(\AA_n)\subseteq \AA_{n-1} \oplus \RR_{t_i}$ and $d_0(\AA_n),d_n(\AA_n)\subseteq \AA_{n-1}.$ Hence, for any $a\in \AA_n$ we obtain $d_i(a)=d_i(a)_{\mathbf{1}} + d_i(a)_{t_i}.$ Therefore, for any $a\in \AA_n$ we get  
\begin{equation}
\begin{split}
\partial_n(a)_{\mathbf{1}} &= d_0(a)+ \sum_{i=1}^{n-1} (-1)^i d_i(a)_{ \mathbf{1} } + (-1)^n d_n(a), \\
\partial_n(a)_{t_i} &= d_i(a)_{t_i}, \hspace{1cm} 1\leq i \leq n-1,
\end{split}
\end{equation}
and $\partial_n(a)_{(k_1,\dots,k_n)}=0$ for all $(k_1,\dots,k_n)\notin \{ \mathbf{1}, t_1,\dots,t_{n-1} \}.$
Therefore $\partial_n(a)\in \AA_n$ if and only if $d_i(a)_{t_i}=0$ for any $1\leq i\leq n.$ In other words, $\partial_n(a)\in \AA_n$ if and only if $d_i(a)=d_i(a)_{\mathbf{1}}$ for any $1\leq i\leq n.$ The assertion follows. 
\end{proof}

\subsection{Spectral sequence of a filtered simplicial set}

For a chain complex $C$ with an exhaustive filtration
\begin{equation}
    0=F_{-1}C \subseteq F_0C \subseteq  F_1C \subseteq \dots \subseteq C, \hspace{1cm} C=\bigcup_{i} F_iC
\end{equation}
there is a natural spectral sequence $E$ converging to $H_*(C)$ such that
\begin{equation}
E^0_{i,j} = F_jC_{i+j}/F_{j-1}C_{i+j}, \hspace{5mm}
E^1_{i,j} = H_{i+j}(F_jC/F_{j-1}C)
\end{equation}
and 
\begin{equation}
E^r_{i,j} = {\sf Im} ( H_{i+j}(F_iC/F_{i-r}C)\longrightarrow H_{i+j}( F_{i+r-1}C/F_{i-1}C)).
\end{equation}
for $r\geq 1$
(see \cite[Th. 5.5]{weibel1995introduction},  \cite[Ch.XV,\S 1,(8)]{cartan1999homological}, \cite[(4.13)]{ivanov2023nested}). For a simplicial set $\mathpzc{X}$ with an exhaustive filtration
\begin{equation}
\emptyset = \mathpzc{X}^{-1} \subseteq \mathpzc{X}^0 \subseteq \mathpzc{X}^1 \subseteq \dots \subseteq \mathpzc{X}, \hspace{1cm} \mathpzc{X}=\bigcup_{i} \mathpzc{X}^i
\end{equation}
we can consider the chain complex $C=N(\mathpzc{X})$ with an exhaustive filtration defined by the images of the embeddings $F_iC={\sf Im}(N(\mathpzc{X}^i) \mono N(\mathpzc{X})).$ Then the associated spectral sequence $E$ satisfies  the following 
\begin{equation}\label{eq:convergence}
E^\infty_{i,j} \Rightarrow H_{i+j}(\mathpzc{X}),
\hspace{5mm}
E^0_{i,j} = \tilde N(\mathpzc{X}^j/\mathpzc{X}^{j-1})_{i+j}, \hspace{5mm}
E^1_{i,j} = \tilde H_{i+j}(\mathpzc{X}^j/\mathpzc{X}^{j-1})
\end{equation}
and
\begin{equation}\label{eq:e^r}
E^r_{i,j} = {\sf Im} ( \tilde H_{i+j}(\mathpzc{X}^i/\mathpzc{X}^{i-r} )\longrightarrow \tilde H_{i+j}( \mathpzc{X}^{i+r-1}/\mathpzc{X}^{i-1})).
\end{equation}
for $r\geq 1.$ This spectral sequence will be called the spectral sequence of the filtered simplicial set $\mathpzc{X}.$ Note that it is natural by the filtered space $\mathpzc{X}.$

\subsection{Magnitude-path spectral sequence}

Here we present a new approach to the spectral sequence constructed by Asao in \cite[Prop. 7.8]{asao2023magnitude} (see also \cite[Remark 46]{hepworth2017categorifying}, \cite[Prop 4.34.]{asao2023magnitude2}). We construct it as the spectral sequence of a filtered simplicial set. 

\begin{proposition}[{cf. \cite[Prop. 7.8]{asao2023magnitude}, \cite[Remark 46]{hepworth2017categorifying}, \cite[Prop 4.34.]{asao2023magnitude2}}] 
Let $X$ be a digraph and
$E(X)$ be the spectral
sequence of the filtered nerve $\SS^* (X).$ 
Then 
\begin{enumerate}
\item $E^r_{i,j}(X) \Rightarrow H_{i+j}(\Pre(X));$
\item $E^0_{i,j}(X)=\MC^{j}_{i+j}(X),$  $ E^1_{i,j}(X) = \MH^{j}_{i+j}(X);$
\item  $E^1_{*,0}(X)\cong \Omega(X)$ and $  E^2_{n,0}(X)=\PH_n(X).$
\end{enumerate}
\end{proposition}
\begin{proof}
It follows from the isomorphism  ${\sf Nrv}({\sf Pre}(X))\cong  \SS(X)$ (Proposition \ref{prop:nrv}), equality $\MC^l_n(X)=\tilde H_n(\MM^l(X))$  
(Theorem \ref{the:magnitude-relative}), and Proposition \ref{th:asao}.
\end{proof}

\begin{proposition}\label{prop:filtered_colimits} The magnitude-path spectral sequence commutes with filtered colimits of digraphs
\[E^r_{*,*}( {\sf colim}(X^\alpha) ) \cong {\sf colim} (E^r_{*,*}(X^\alpha)). \]
In particular, magnitude and path homology commute with filtered colimits of digraphs.
\end{proposition}
\begin{proof}
Let $X^\alpha$ be a filtered diagram of digraphs (a functor from a filtered category) and $X={\sf colim}(X^\alpha)$. We denote by $\iota^\alpha:X^\alpha \to X$ the canonical maps. It is easy to see that $X_0={\sf colim}(X^\alpha_0)$ and $X_1={\sf colim}(X^\alpha_1).$ Therefore, any path of length $n$ in $X$ can be lifted to a path of length $n$ in $X^\alpha$ for some $\alpha.$ Hence ${\sf Pre}(X)\cong {\sf colim}\: {\sf Pre}(X^\alpha).$

Any finite category is a compact object in a category of small categories. In particular, the poset $[n]=\{0<1<\dots<n\}$ treated as a category is a compact object in the category of small 
categories. Since ${\sf Nrv}_n(\CCC) = {\sf Cat}([n],\CCC),$ 
we obtain that the functor of nerve ${\sf Nrv}:{\sf Cat} \to {\sf sSet}$ commutes with filtered colimits. Hence  ${\sf Pre}(X)\cong {\sf colim}\: {\sf Pre}(X^\alpha)$ implies  $\SS(X)\cong {\sf colim}\: \SS(X^\alpha).$

The simplicial set ${\sf colim}\: \SS(X^\alpha)$ has a  filtration defined by the images ${\sf colim}\: \SS^\ell(X^\alpha) \to {\sf colim}\: \SS(X^\alpha).$ Since any path of length $n$ in $X$ can be lifted to a path of length $n$ in $X^\alpha$ for some $\alpha,$ we obtain that any simplex from $\SS^\ell(X)$ can be lifted to a simplex of $\SS^\ell(X^\alpha)$ for some $\alpha.$ Therefore we obtain an isomorphism of filtered simplicial sets $\SS^*(X)\cong {\sf colim}\: \SS^*(X^\alpha).$ Using that filtered colimits commute with pushouts, we obtain 
\begin{equation}
\SS^\ell(X)/\SS^{\ell'}(X) \cong {\sf colim}\: \SS^\ell(X^\alpha)/\SS^{\ell'}(X^\alpha)  
\end{equation}
for any $\ell\geq \ell'.$ Then the assertion follows from the formula   \eqref{eq:e^r} and the fact that the homology of a simplicial set commute with filtered colimits. 
\end{proof}

\begin{corollary}\label{cor:colimit}
If $X$ is a digraph and there is an increasing  sequence of subdigraphs
$ X^1 \subseteq X^2 \subseteq \dots \subseteq X, $ such that $ X=\bigcup_{n\geq 0} X^n,$
then $\PH_*(X)={\sf colim}\: \PH_*(X^n).$
\end{corollary}

\section{\bf Fundamental groupoids of digraphs}

\subsection{Category defined by a quiver with relations}

When we work with small categories and groupoids, we prefer to write composition from the left to the right. In this case we use the dot 
 \begin{equation}
 \begin{tikzcd}
    & \bullet \ar[rd,"g"] &\\
   \bullet \ar[ru,"f"] \ar[rr,"{f\cdot g = g f}"'] & & \bullet
 \end{tikzcd}
 \end{equation}

By a quiver we mean a $5$-tuple $Q=(Q_0,Q_1,d_0,d_1,s_0),$ where $Q_0$ and $Q_1$ are sets called ``the set of vertices'' and  ``the set of arrows'', $d_0,d_1:Q_1\to Q_0$ are maps called  ``head'' and  ``tail'' of an arrow, and $s_0:Q_0\to Q_1$ is a map called ``degeneracy map'', which satisfy
$d_0  s_0=d_1  s_0={\sf id}_{Q_0}.$
Arrows from $s_0(v)$ are called degenerate  loops, all other arrows are called non-degenerate arrows. 
Since we think about quivers as about ``categories without composition'' we will also use the notation 
\begin{equation}
    1_v = s_0(v), \hspace{5mm} d_0(\alpha)={\sf cod}(\alpha), \hspace{5mm} d_1(\alpha)={\sf dom}(\alpha).
\end{equation}
A quiver can be equivalently defined as a 1-truncated simplicial set. A morphism of quivers $f:Q\to Q'$ is a couple of maps $f_0:Q_0\to Q'_0$ and $f_1:Q_1\to Q'_1$ such that $d_if_1 =f_0d_i$ and $s_0f_0=f_1s_0.$ 

We define a path in a quiver $Q$ from a vertex $v$ to a vertex $u$ is a sequence of arrows $(\alpha_1,\dots,\alpha_n)$ such that $d_1(\alpha_1)=v, d_0(\alpha_i)=d_1(\alpha_{i+1}), d_0(\alpha_n)=u.$ We consider an equivalence relation on the set of paths, which is the minimal equivalence relation such that $(\alpha_1,\dots,\alpha_n) \  \sim \  (\alpha_1, \dots, \alpha_i, 1_{v_i}, \alpha_{i+1}, \dots, \alpha_n).
$ 
A class of paths containing  $(\alpha_1,\dots,\alpha_n)$ 
will be denoted by 
$\alpha_1\cdot {\dots} \cdot \alpha_n.$ 
Any such a class can be rewritten either as $1_v$ for some $v,$ or as 
$\alpha_1\cdot {\dots} \cdot \alpha_n,$ 
where $\alpha_i$'s are  non-degenerate arrows. The free category $F^{\sf cat}(Q)$ is the category whose set of objects is $Q_0$ and morphisms from $v$ to $u$ are equivalence classes of paths. The composition is defined by the concatenation, the identity morphism is defined as $1_v=s(v).$ This category satisfies the obvious universal property. One can say that $F^{\sf cat}$ is the left adjoint functor to the forgetful functor $U:{\sf Cat}\to {\sf Quiv}.$

Assume that $\CCC$ is a small category and $\RRR\subseteq {\sf Mor}(\CCC)^2$ is a set of parallel pairs of morphisms i.e. for any $(r,r')\in \RRR$ we have ${\sf dom}(r)={\sf dom}(r')$ and ${\sf cod}(r)={\sf cod}(r')$. Then we denote by $\CCC/\RRR$ the quotient category $\CCC/\!\sim,$ where $\sim$ is the minimal congruence relation such that $r\sim r'$ for any $(r,r')\in \RRR.$ The quotient category satisfies the following universal property. For any functor $\tilde \Phi:\CCC \to \mathpzc{D}$ such that $\tilde \Phi(r)=\tilde \Psi(r')$ for any $(r,r')\in \RRR,$ there exists a unique functor $\Phi:\CCC/\RRR \to \mathpzc{D}$ such that $\Phi \pi = \tilde \Phi,$ where $\pi:\CCC\to \CCC/\RRR$ is the canonical projection.

If $\CCC=F^{\sf cat}(Q),$ the quotient category  
\begin{equation}
\langle Q\mid \RRR \rangle_{\sf cat} := F^{\sf cat}(Q)/\RRR. 
\end{equation}
is called the category defined by the quiver $Q$ with relations $\RRR.$

\subsection{Groupoid defined by a quiver with relations}

 The free groupoid $F(Q)$ generated by a quiver $Q$ is defined as the localization of the free category by all morphisms
\begin{equation}
F(Q) = F^{\sf cat}(Q) [{\sf Mor}(F^{\sf cat}(Q))^{-1}]. 
\end{equation}
It is easy to see that free groupoid satisfies the obvious universal property.

Equivalently the free groupoid can be defined as follows. Consider the quiver $Q^{\pm}$ such that $Q^{\pm}_0=Q_0$ and $Q^{\pm}_1$ consists of elements of two types: $1_v$ for $v\in Q_0;$ formal elements $\alpha^\varepsilon,$ where $\alpha\in Q_1\setminus s_0(Q_0)$ and $\varepsilon\in\{-1,1\}.$ The maps $d_0,d_1,s_0$ are defined by $s_0(v)=1_v,$ $d_k(1_v)=v,$ $d_k(\alpha^1)=d_k(\alpha)$ and $d_k(\alpha^{-1})=d_{1-k}(\alpha).$ Then it is easy to see that
\begin{equation}
F(Q) \cong \langle Q^{\pm}\mid \alpha^\varepsilon \cdot \alpha^{-\varepsilon} = 1_{d_0(\alpha^\varepsilon)} \rangle_{\sf cat}.
\end{equation}
If $X$ is a digraph, we can consider the associated quiver $Q(X),$ where $Q(X)_0=X_0,$ $Q(X)_1=X_1,$ $s_0(x)=(x,x),$ $d_0(x_0,x_1)=x_1,$ $d_1(x_0,x_1)=x_0$ and set $F(X)=F(Q(X)).$

If $\RRR\subseteq {\sf Mor}(F(Q))^2$ is a set of parallel pairs of morphisms, we set 
$\langle Q\mid \RRR \rangle = F(Q)/\RRR.$
The groupoid $\langle Q\mid \RRR \rangle$ is called the groupoid defined by the quiver $Q$ with relations $\RRR.$ If $X$ is a digraph, we also set 
\begin{equation}
\langle X\mid \RRR \rangle = F(X)/\RRR, 
\end{equation}
and call this groupoid a groupoid defined by a digraph with relations.

\subsection{Fundamental groupoid of a simplicial set}

For a simplicial set $\mathpzc{X}$ the fundamental  groupoid $\Pi(\mathpzc{X})$ is the groupoid defined by the quiver 
\begin{equation}
 Q(\mathpzc{X})=(\mathpzc{X}_0,\mathpzc{X}_1,d_0,d_1,s_0)   
\end{equation}
and relations 
\begin{equation}\label{relations:grouppoid}
d_2(z)\cdot  d_0(z)= d_1(z) 
\end{equation}
for any $z\in \mathpzc{X}_2.$ Note that by the definition of a free groupoid we also have 
$s_0(x)=1_x$ for any $x\in \mathpzc{X}_0.$

\subsection{\textbf{\textit{l}}-fundamental groupoid of a digraph} 

In this subsection we will prove that for a digraph $X$ the fundamental groupoid of $\SS^l(X)$ is the groupoid defined by the digraph $X$, where for any two vertices $x,x'$ all paths from $x$ to $x'$ of lengths $\leq l$ are identified. We will call $\Pi(\SS^l(X))$ the $l$-fundamental groupoid of $X$ and denote it by
\[ \Pi^l(X):=\Pi(\SS^l(X)). \]

\begin{lemma}\label{lemma:an_equation_in_pi^l}
For any path $(x_0,\dots,x_n)$ in $X$ such that $0\leq n\leq l$ we have the following equation in $\Pi^l(X)$
\begin{equation}
(x_0,x_1)\cdot (x_1,x_2) \cdot {\dots} \cdot (x_{n-1},x_n) = (x_0,x_n).
\end{equation}
\end{lemma}
\begin{proof}
This can be proved by induction on $n.$ For $n=0,1$ it is obvious. The inductive step follows from the computation
\begin{equation}
\begin{split}
(x_0,x_1)\cdot (x_1,x_2) \cdot {\dots} \cdot (x_{n-1},x_n) &= (x_0,x_{n-1})\cdot (x_{n-1},x_n) \\
& =d_2(x_0,x_{n-1},x_n)\cdot d_0(x_0,x_{n-1},x_n)\\
& = d_1(x_0,x_{n-1},x_n) = (x_0,x_n).
\end{split}
\end{equation}
\end{proof}

\begin{proposition}\label{prop:pi^l} For any digraph $X$
the $l$-fundamental groupoid  $\Pi^l(X)$ is a groupoid naturally isomorphic to a groupoid defined by the digraph $X$ with relations 
\begin{equation}\label{eq:rel_l}
(x_0,x_1)\cdot (x_1,x_2)\cdot {\dots} \cdot (x_{n-1},x_n) = (y_0,y_1)\cdot (y_1,y_2)\cdot {\dots} \cdot (y_{m-1},y_m)     
\end{equation}
where $x_0=y_0,$ $x_n=y_m,$ $(x_i,x_{i+1}),(y_j,y_{j+1})\in  X_1$ and $0\leq n,m\leq l.$ 
\end{proposition}
\begin{proof} For the proof we denote by $\mathpzc{G}$ the groupoid defined by $X$ with relations \eqref{eq:rel_l} and $\Pi^l:=\Pi^l(X)$. We need to prove that $\mathpzc{G}\cong \Pi^l.$ The plan of the proof is the following. We will construct two functors $\Phi: \mathpzc{G}\to \Pi^l$ and $\Psi: \Pi^l \to \mathpzc{G}$ and prove that $\Psi \Phi={\sf Id}$ and $ \Phi\Psi = {\sf Id}.$ In order to construct the functors $\Phi$ and $\Psi,$ we first construct functors $ \tilde \Phi:F(X) \to \Pi^l$ and $\tilde \Psi: F(Q(\SS^l(X))) \to \mathpzc{G}.$

Since $Q(X)\subseteq Q(\SS^l(X)),$ we
can consider a functor $\tilde \Phi:F(Q(X))\to \Pi^l$ which is identical on vertices and arrows in $Q(X).$ Lemma \ref{lemma:an_equation_in_pi^l} implies that the relations \eqref{eq:rel_l} hold in $\Pi^l$. 
Therefore, $\tilde \Psi$ induces a functor $\Phi:\mathpzc{G} \to \Pi^l$.  

The set $Q(\SS^l(X))_1$ consists of couples $(x,x')$ such that $\dist(x,x')\leq l.$
We define a functor $\tilde \Psi : F(Q(\SS^l(X))) \to \mathpzc{G}$ on the quiver $Q(\SS^l(X))$ so that it is identical on vertices and sends an arrow $(x,x')$ to a path from $x$ to $x'$ of length $\leq l:$ 
\begin{equation}
\tilde \Psi( (x,x') ) = (y_0,y_1)\cdot (y_1,y_2)\cdot {\dots} \cdot (y_{n-1},y_n),
\end{equation}
where $y_0=x,y_n=x', \dist(y_i,y_{i+1})\leq 1$ and $0\leq n\leq l.$ Relations \eqref{eq:rel_l} show that the definition does not depend on the choice of the path. In order to prove that $\tilde \Psi$ induces a functor $\Psi: \Pi^l\to \mathpzc{G},$ we need to check that
\begin{equation}
\tilde \Psi(d_2((x_0,x_1,x_2))\cdot d_0((x_0 , x_1 , x_2)))= \tilde \Psi ( d_1((x_0,x_1,x_2))) 
\end{equation}
for any $(x_0,x_1,x_2)\in \SS^l(X)_2.$ We set $n:= \dist(x_0,x_1)$ and $m:= \dist( x_1,x_2).$ Then $n+m\leq l.$ Consider a path $(y_0,\dots,y_n)$ from $x_0=y_0$ to $x_1=y_n$, and a path $(y_n,\dots,y_{n+m})$ from $x_1=y_n$ to $x_2=y_{n+m}.$ Then 
\begin{equation}
\begin{split}
\tilde \Psi(d_2((x_0,x_1,x_2))\cdot d_0((x_0 , x_1 , x_2))) &= \tilde \Psi((x_0,x_1)) \cdot \tilde \Psi((x_1,x_2)) \\ 
 =((y_0,y_1)\cdot {\dots} \cdot (y_{n-1},y_n)) &\cdot ((y_n,y_{n+1})\cdot {\dots} \cdot (y_{n+m-1},y_{n+m}))\\
& = (y_0,y_1)\cdot {\dots} \cdot (y_{n+m-1},y_{n+m})\\
& = \tilde \Psi((x_0,x_2)) = \tilde \Psi(d_1((x_0,x_1,x_2))).
\end{split}
\end{equation}
Therefore $\tilde \Psi$ induces a functor $\Psi: \Pi^l \to \mathpzc{G}.$ 

Let us check the identities $\Psi \Phi = {\sf Id}_{\mathpzc{G}}$ and $\Phi \Psi = {\sf Id}_{\Pi^l}.$ It is sufficient to check them only for arrows from the generating quivers. The identity $\Psi(\Phi((x,x'))) = (x,x')$ is obvious for any arrow $ (x,x')$. For any $(x,x')\in Q(\SS^l(X))_1$ we take a path $(y_0,\dots,y_n)$ such that $x=y_0,x'=y_n, (x,x')\in X_1$ and $n\leq l.$ Using Lemma \ref{lemma:an_equation_in_pi^l}, we obtain 
$\Phi(\Psi(x,x')) = (y_0,y_1)\cdot {\dots} \cdot (y_{n-1},y_n) = (x,x').$
\end{proof}

\subsection{Grigor’yan-Jimenez-Muranov groupoid} 
Grigor’yan, 
Jimenez and Muranov in \cite{grigor2018fundamental} define a version of the fundamental groupoid of a digraph that we denote by $\Pi^{\sf GJM}(X)$. They show that the corresponding version of the fundamental group coincides with the fundamental group defined by Grygor'yan, Lin, Muranov and Yau in \cite{grigor2014homotopy}.  In their definition they use the notion of a square and a triangle in a digraph. A \emph{square} in $X$ is a quadruple of distinct vertices $(x,y,y',z)$ such that $\dist(x,y)=\dist(x,y')=\dist(y,z)=\dist(y',z)=1.$ A \emph{triangle} in $X$ is a triple of distinct vertices $(x,y,z)$ such that $ \dist( x,y) = \dist(y,z) = \dist(x,z)=1.$ Their definition can be rewritten as follows. For a digraph $X$ we consider a digraph $X^{\pm}$ such that  $X^{\pm}_0=X_0$ and $(x,x')\in X^{\pm}_1$ if and only if $(x,x')\in  X_1$ or $(x',x)\in X_1.$ 
Then the groupoid $\Pi^{\sf GJM}(X)$ is the groupoid defined by the digraph $X^{\pm}$ and relations 
\begin{description}
\item[\sf (GJM1')] $(x,y)\cdot (y,x) = {\sf id}_x;$
\item[\sf (GJM2')] $(x,y)\cdot (y,z) = (x,z),$ if $(\sigma(x),\sigma(y),\sigma(z)) $ is a triangle for some permutation $\sigma;$
\item[\sf (GJM3')] $(x,y) \cdot (y,z) = (x,y') \cdot (y',z)$ if $(\sigma(x),\sigma(y),\sigma(y'),\sigma(z))$ is a square, where  $\sigma$ is a power of the cyclic permutation $x\to y\to z\to y'\to x.$
\item[\sf (GJM4')] $(x,y) \cdot (y,z) \cdot (z,y') = (x,y')$ if $(\sigma(x),\sigma(y),\sigma(y'),\sigma(z))$ is a square, where  $\sigma$ is a power of the cyclic permutation $x\to y\to z\to y'\to x;$
\item[\sf (GJM5')] $(x,x)=1_x.$
\end{description}

\begin{lemma} The relations {\sf (GJM1')-(GJM5')} are equivalent to the following smaller set of relations:
\begin{description}
\item[\sf (GJM1)] $(x,y)\cdot (y,x) = {\sf id}_x;$
\item[\sf (GJM2)] $(x,y)\cdot (y,z) = (x,z),$ if $(x,y,z)$ is a triangle;
\item[\sf (GJM3)] $(x,y) \cdot (y,z) = (x,y') \cdot (y',z),$ if  $(x,y,y',z)$ is a square.
\end{description}
\end{lemma}
\begin{proof}
First note that our definition of the free groupoid we have $(x,x)=s(x)={\sf id}_x.$ Hence we can skip \textsf{(GJM5')}. 
It is also easy to see that \textsf{(GJM1')} and \textsf{(GJM3')} imply \textsf{(GJM4')} 
\begin{equation}
(x,y)\cdot (y,z)\cdot (z,y') = (x,y')\cdot (y',z)\cdot (z,y') = (x,y')    
\end{equation}
if  $(\sigma(x),\sigma(y),\sigma(y'),\sigma(z))$ is a square, where  $\sigma$ is a power of the cyclic permutation $x\to y\to z\to y'\to x.$  Therefore, we can skip \textsf{(GJM4')}. Also note that {\sf (GJM1)=(GJM1')}.
So we need to prove that {\sf (GJM1)-(GJM3)} imply {\sf (GJM2'),(GJM3')}.

Let us prove that {\sf (GJM1),(GJM2)} imply {\sf (GJM2')}.  For any triangle $(x,y,z)$  the relation $(x,y)\cdot (y,z)=(x,z)$ combined with \textsf{(GJM1')} implies five permuted relations
\begin{equation*}
\begin{split}
(y,x)\cdot (x,z) &= (y,x) \cdot (x,y) \cdot (y,z) = (y,z),\\
 (x,z)\cdot (z,y) &= (x,y)\cdot (y,z)\cdot (z,y) = (x,y),\\
 (z,y)\cdot (y,x) &= ((x,y)\cdot (y,z))^{-1} = (x,z)^{-1} = (z,x),\\ 
 (z,x)\cdot (x,y) &= ( (y,x)\cdot (x,z) )^{-1} = (y,z)^{-1} = (z,y),\\
 (y,z)\cdot (z,x) &= ((x,z)\cdot (z,y))^{-1} = (x,y)^{-1} = (y,x). 
\end{split}
\end{equation*}

Let us prove that {\sf (GJM1),(GJM3)} imply {\sf (GJM3')}. If $(x,y,y',z)$ is a square the relation $(x,y)\cdot (y,z) = (x,y')\cdot (y',z)$ combined with  {\sf (GJM1)} implies three permuted relations
\begin{equation*}
\begin{split}
(y,z) \cdot (z,y') &= (y,x)\cdot (x,y) \cdot (y,z) \cdot (z,y') \\ &= (y,x)\cdot (x,y') \cdot (y',z) \cdot (z,y') =  (y,x)\cdot (x,y'); \\
(z,y') \cdot (y',x) & = ( (x,y') \cdot (y',z) )^{-1}  =( (x,y) \cdot (y,z) )^{-1} = (z,y)\cdot (y,x); \\
(y',x)\cdot (x,y) &= ( (y,x)\cdot (x,y') )^{-1} = ( (y,z)\cdot (z,y') )^{-1} = (y',z)\cdot (z,y). 
\end{split}
\end{equation*}
The assertion follows.
\end{proof}

\begin{theorem}\label{th:GJM} The 
Grigor’yan-Jimenez-Muranov  fundamental groupoid is naturally isomorphic to the $2$-fundamental groupoid 
\begin{equation}
\Pi^{\sf GJM}(X) \cong \Pi^2(X).
\end{equation}
\end{theorem}
\begin{remark}
Grigor’yan, Jimenez and Muranov define a CW-complex $\Delta(X)$ associated with $X$, whose homotopy groupoid coincides with $\Pi^{\sf GJM}(X).$ In some sense $\SS^2(X)$ is similar to $\Delta(X).$ However, in general $|\SS^2(X)|$ is not homotopy equivalent to $\Delta(X).$ For example, one can take 
$\begin{tikzcd}[row sep=-1mm, column sep=5mm]
& \cdot\ar[rd] & 
\\
\cdot \ar[r] \ar[ru] \ar[rd] & \cdot \ar[r] & \cdot 
\\
& \cdot \ar[ru] & 
\end{tikzcd}$. 
In this example $\Delta(X)\sim S^2$ and $|\NN^2(X)|$ is contractible. 
\end{remark}
\begin{proof}[Proof of Theorem \ref{th:GJM}] Set $\Pi^2:=\Pi^2(X)$ and $\Pi^{\sf GJM}:=\Pi^{\sf GJM}(X).$ Here we will use the description of $\Pi^2$ given in Proposition \ref{prop:pi^l}. So $\Pi^2$ is the groupoid defined by the digraph $X$ and relations \eqref{eq:rel_l} for $l=2.$ In other words, all paths with the same tail and head of length $\leq 2$ are identified. In order to prove the statement we will construct functors $\Phi:\Pi^{\sf GJM}\to \Pi^2$ and $\Psi:\Pi^2 \to \Pi^{\sf GJM}$ such that $\Phi\circ \Psi={\sf Id}$ and $\Psi\circ \Phi = {\sf Id}.$ 

The digraph $X$ is a sub-digraph of $X^{\pm}.$ The embedding induces a functor $\tilde \Psi : F(X) \to \Pi^{\sf GJM}(X).$ In order to prove that $\tilde \Psi$ induces a functor $\Psi: \Pi^2 \to \Pi^{\sf GJM}$ we need to check that the relations  \eqref{eq:rel_l} are satisfied in $\Pi^{\sf GJM}$ for $0\leq n,m\leq 2.$ Without loss of generality we can assume that $n\leq m.$ If $n=m=2,$ this follows from \textsf{(GJM3)}. If $n=1,m=2,$ this follows from \textsf{(GJM2)}. If $n=m=1,$ it is obvious. If $n=0,m=2,$ this follows from \textsf{(GJM1)}. If $n=0,m=1,$ then $x_0=x_1=y_0=y_1$ and it follows from the equation $1_{x_0}=(x_0,x_0).$ If $n=m=0,$ it is obvious. Therefore, $\tilde \Psi $ induces a functor $\Psi : \Pi^2 \to \Pi^{\sf GJM}.$

Consider a functor 
$\tilde \Phi: 
 F(X^{\pm}) \to 
 \Pi^2$
which is identical on objects and defined on arrows of $X^{\pm}$ by the formula
\begin{equation}
\tilde \Phi((x,x')) = \begin{cases}
(x,x') & \text{ if } (x,x')\in  X_1 \\
(x',x)^{-1} & \text{ if } (x',x) \in  X_1 \text{ and } (x,x')\notin  X_1.
\end{cases}
\end{equation}
In order to prove that $\tilde \Phi$ induces a functor $\Phi:\Pi^{\sf GJM} \to \Pi^2,$ we need to prove that $\tilde \Phi$ sends all the relations \textsf{(GJM1)}-\textsf{(GJM3)} to identities.  

\textsf{(GJM1)}. We need to consider three cases. 

\textsf{(GJM1)}.1. Assume that $(x,y),(y,x)\in X_1.$ Then $(x,y)\cdot (y,x)=(x,x)$ in $\Pi^2$ and 
\begin{equation*}
 \tilde \Phi((x,y))\cdot \tilde \Phi((y,x)) = (x,y)\cdot (y,x)  = (x,x)=1_x.   
\end{equation*}

\textsf{(GJM1)}.2. Assume that $(x,y)\in X_1$ and $(y,x)\notin X_1.$ Then 
\begin{equation*}
\tilde \Phi((x,y)\cdot (y,x)) = (x,y)\cdot (x,y)^{-1} =1_x. 
\end{equation*}

\textsf{(GJM1)}.3. The case $(x,y)\notin X_1$ and $(y,x)\in X_1$ is similar to \textsf{(GJM1)}.2.

\textsf{(GJM2)}. If $(x,y,z)$ is a triangle, then  $(x,y) \cdot (y,z)=(x,y)$ in $\Pi^2$ and 
\begin{equation*}
\tilde \Phi((x,y)\cdot (y,z)) = (x,y)\cdot (y,z)=(x,z) = \tilde \Phi((x,z)). 
\end{equation*}

\textsf{(GJM3)} Assume that $(x,y,y',z)$ is a square. Then  $(x,y)\cdot (y,z)=(x,z)=(x,y')\cdot (y',z)$ in $\Pi^2$ and 
\begin{equation*}
\tilde \Phi((x,y)\cdot (y,z)) = (x,y) \cdot (y,z) = (x,y')\cdot (y',z) = \tilde \Phi((x,y')\cdot (y',z)).
\end{equation*}
Therefore, $\tilde \Phi$ induces a well-defined functor
$\Phi : \Pi^{\sf GJM} \to \Pi^2.$

Now we prove that $\Phi\circ \Psi = {\sf Id}_{\Pi^2}$ and $\Psi \circ  \Phi={\sf Id}_{\Pi^{\sf GJM}}.$ It is enough to prove these equations for arrows from generating quivers. For any $(x,x')\in X_1$ the equation $ \Phi(\Psi((x,x'))) = (x,x')$ is obvious. Hence, $\Phi\circ \Psi={\sf Id}_{\Pi^{\sf GJM}}.$ For any $(x,x')\in X_1$ the equation $ \Psi(\Phi((x,x'))) = (x,x')$ is also obvious. If $(x,x')\notin X_1$ but $(x',x)\in X_1,$ then $\Psi(\Phi((x,x'))) = \Psi((x',x)^{-1})=(x',x)^{-1}=(x,x')$ by \textsf{(GJM1)}.
\end{proof}

\subsection{Clusters in path chains} Let $X$ be a digraph and $\RR(X)$ be the chain complex of regular accessible sequences. For two vertices $x,x'$ we denote by 
$\AA_n(X,x,x')$ the submodule of $\AA_n(X)$ generated by paths with tail $x$ and head $x'.$ Then 
\begin{equation}
\AA_n(X)=\bigoplus_{(x,x')\in X_0^2} \AA_n(X,x,x').    
\end{equation}
It is known that this decomposition induces a decomposition for the components of the complex of path chains
\begin{equation}\label{eq:clusters}
\Omega_n(X) = \bigoplus_{(x,x')\in X_0^2} \Omega_n(X,x,x'),
\end{equation}
where $\Omega_n(X,x,x')=\Omega_n(X)\cap \AA_n(X,x,x')$ \cite[Lemma 2.2]{grigor2022advances}. Elements of $\Omega_n(X,x,x')$ are called clusters or $(x,x')$-clusters. 

In this subsection we want to present a finer decomposition of $\Omega_n(X)$, which is indexed by morphisms of the $2$-fundamental groupoid. For a morphism $\gamma$ in $\Pi^2(X)$ we denote by $\AA_n(X,\gamma)$ the submodule of $\AA_n(X)$ generated by paths that represent $\gamma.$ Then 
\begin{equation}
 \AA_n(X,x,x') = \bigoplus_{\gamma\in \Pi^2(X)(x,x')} \AA_n(X,\gamma).   
\end{equation}
We consider the module $\Omega_n(X,\gamma)=\Omega_n(X)\cap \AA_n(X,\gamma),$ whose elements will be called $\gamma$-clusters.  

\begin{proposition}\label{prop:clusters} 
Any $(x,x')$-cluster can be uniquely presented as a sum of $\gamma$-clusters
\begin{equation}
\Omega_n(X,x,x') = \bigoplus_{\gamma\in \Pi^2(X)(x,x')} \Omega_n(X,\gamma).
\end{equation}
\end{proposition}
\begin{proof}
For an element $r\in \RR_n(X)$ and an accessible sequence $(x_0,\dots,x_n)$ we denote by $r^{x_0,\dots,x_n}\in \KK$ the coefficient of $r$ at the sequence $(x_0,\dots,x_n)$. If $(x_0,\dots,x_n)$ is not regular, we assume $r^{x_0,\dots,x_n}=0.$ Therefore, $r\in \AA_{n}(X)$ if and only if for any accessible sequence $(x_0,\dots,x_n),$ which is not a path, we have $r^{x_0,\dots,x_n}=0.$

For an element $a\in \AA_n(X,x,x')$ we denote by $a_\gamma$ its component in $\AA_n(X,\gamma).$ Its coefficients are defined by the formula
\begin{equation}
 a_\gamma^{x_0,\dots,x_n} = 
 \begin{cases}
 a^{x_0,\dots,x_n}, & \text{ if } (x_0,\dots,x_n) \text{ represents } \gamma,\\
 0, & \text{ else. } 
 \end{cases}   
\end{equation}

Assume that $a\in \Omega_n(X,x,x')$ and $\gamma\in \Pi^2(X)(x,x').$ It is sufficient to prove that $a_\gamma\in \Omega_n(X).$ Here we use Proposition \ref{prop:description_of_PC} for the description of $\Omega_n(X).$ Then $d_i(a)\in \AA_{n-1}(X)$ for any $1\leq i\leq n-1.$ Consider an accessible sequence $(y_0,\dots,y_{n-1}),$ which is not a path. Therefore we know that $d_i(a)^{y_0,\dots,y_{n-1}}=0,$ and we need to prove that $d_i(a_\gamma)^{y_0,\dots,y_{n-1}}=0$ for any $1\leq i\leq n-1.$ By the formula for $d_i$ we see that
\begin{equation}\label{eq:sum_d_i}
0=d_i(a)^{y_0,\dots,y_{n-1}}=\sum_{z\in X_0} a^{y_0,\dots, y_{i-1},z,y_i,\dots,y_{n-1}}.
\end{equation}
However, if $\sigma=(y_0,\dots,y_{i-1},z,y_{i+1},\dots,y_n)$ is not a regular path, then 
$a^{\sigma} =0.$ Therefore, if we denote by $T$ 
(here $T$ depends on $(y_0,\dots,y_n)$ and $i$) 
the set of all regular paths of the form $(y_0,\dots ,y_{i-1},$ $z,y_i,\dots,y_n)$, where $z\in X_0,$ we obtain
$\sum_{t\in T} a^{t} = 0.$
By Proposition \ref{prop:pi^l}
for any $z,z'$ such that $(y_{i-1},z,y_i)$ and $(y_{i-1},z',y_i)$ are paths,  the paths $(y_{i-1},z,y_i)$ and $(y_{i-1},z',y_i)$ represent the same element in $\Pi^2(X).$ Hence, any two elements $t_1,t_2\in T$ also represent the same element in $\Pi^2(X)$. If all elements of $T$ represent the element $\gamma,$ then $a_\gamma^{t}=a^{t}$ for all $t\in T.$ If all elements of $T$ represent some other element $\gamma'\ne \gamma$, then $a_\gamma^{t}=0$ for all $t\in T.$ In both cases we obtain 
$d_i(a_\gamma)^{y_0,\dots,y_{n-1}}=\sum_{t\in T} a_\gamma^{t}=0.$
\end{proof}

\section{\bf Covering  digraphs}

\subsection{Covering quivers} 

Denote by $*$ the one point quiver with vertex $0$, by $I$ the quiver with two vertices $0,1$ and one non-degenerate arrow $0\to 1.$ There are two morphisms $\iota_k:*\to I$  such that $\iota_k(0)=k$ for $k\in \{0,1\}.$ A morphism of quivers $p:E\to Q$ is called covering of $X$, if for any commutative diagram 
\begin{equation}
\begin{tikzcd}
* \ar[r] \ar[d,"i_k"] & E \ar[d,"p"] \\
I \ar[r] & Q
\end{tikzcd}
\end{equation}
(with any $k\in \{0,1\}$) there exists a unique morphism $I\to E$ making the diagram commutative. In other words, $p$ is covering, if for any arrow of $Q$ a lift of its tail or head uniquely defines a lift of the arrow. To be more precise: $p$ is a covering if and only if for any $\gamma\in Q_1$ and any $e\in p^{-1}(d_1(\gamma))$ (resp. $e'\in p^{-1}(d_0(\gamma) )$) there exists a unique $\tilde \gamma \in  E_1$ such that $p(\tilde \gamma)=\gamma$ and $d_1(\tilde\gamma) = e$ (resp. $d_0(\tilde\gamma) = e'$).

For subsets $W,W'$ of vertices of a  quiver $Q$ we set 
\begin{equation}
Q(W,W') = \{ \gamma\in Q_1\mid d_1(\gamma) \in W, d_0(\gamma)\in W' \},    
\end{equation}
and if $W=\{q\}$ or $W'=\{q\},$ we set $Q(q,W')=Q(\{q\},W')$ and $ Q(W,q)=Q(W,\{q\}).$  Then the property of being covering for $p$ can be also reformulated as follows: for any $e\in E_0$ and $q\in Q_0$ the maps
\begin{equation}
 E(e,p^{-1}(q)) \to Q(p(e),q), \hspace{1cm}  E(p^{-1}(q),e) \to Q(q,p(e))
\end{equation}
are bijections.

\begin{lemma}\label{lemma:sub_1_covering}
Let $p:E\to Q$ be a covering quiver and $Q'\subseteq Q$ is a subquiver. Then the restriction $p^{-1}(Q')\to Q'$ is also a covering quiver.   
\end{lemma}
\begin{proof}
Obvious. 
\end{proof}

\begin{lemma}\label{lemma:path_lifting}
Let $p:E\to Q$ be a covering of quivers and $(\gamma_1,\dots,\gamma_n)$ be a path in $Q$ and $(q_0,\dots,q_n)$ be the associated sequence of vertices in $Q_0$. Then for any $0\leq k\leq n$ and $e\in p^{-1}(q_k)$ there exists a unique path $(\tilde \gamma_1,\dots,\tilde \gamma_n)$ in $E$ with the associated sequence of vertices $(e_0,\dots,e_n)$ such that $p(\tilde \gamma_i)=\gamma_i$ for any $0\leq i\leq n$ and $e_k=e.$
\end{lemma}
\begin{proof}
Obvious. 
\end{proof}

\subsection{\textit{l}-covering digraphs}

A morphism of digraphs $p:E\to X$ is called $1$-covering, if $Q(p):Q(E)\to Q(E)$ is a covering quiver.

For a natural $n$ we denote by $D_n(X)$ the digraph such that $D_n(X)_0=X_0$ and $D_n(X)_1=\{(x,x')\mid \dist(x,x')\leq n\}.$ In particular, $D_1(X)=X.$ The digraph $D_n(X)$ is natural by $X$ and we obtain a functor
\begin{equation}
 D_n: {\sf Digr}\longrightarrow {\sf Digr}.   
\end{equation}

For a natural $l$ we say that a morphism of digraphs $p:E\to X$ is an \emph{$l$-covering}, if $D_n(p):D_n(E)\to D_n(X)$ is a $1$-covering for any $1\leq n\leq l.$

\begin{proposition}
\label{prop:covering_description1} 
For a morphism of digraphs $p:E\to X$ the following statements are equivalent.
\begin{enumerate}
\item $p$ is an $l$-covering of digraphs.

\item $D_l(p):D_l(E) \to D_l(X)$ is a $1$-covering and $E=D_l(p)^{-1}(X)$. 

\item $p$ is a $1$-covering and, if  $(e_0,\dots,e_n)$ and $(e'_0,\dots,e'_{m})$ are paths in $E$ such that $0\leq n,m\leq l$ and either $e_0=e'_0,$ $p(e_n)=p(e_{m}')$ or $p(e_0)=p(e_0'),$ $e_n=e'_{m},$ then $e_0=e'_0, e_n=e'_{m}.$

\item For any vertices $x,x'$ of $X$ such that $0\leq \dist(x,x')\leq l$ there is a unique bijection $\alpha:p^{-1}(x)\to p^{-1}(x')$ such that for any $e\in p^{-1}(x)$ we have $\dist(e,\alpha(e)) = \dist(x,x')$  and $\dist(e,e')>l$ for any $e'\in p^{-1}(x')\setminus\{\alpha(e)\}.$
\end{enumerate}

\end{proposition}
\begin{proof} 
$(1)\Rightarrow (2).$ By the assumption $D_l(p)$ is a $1$-covering. Obviously, $E\subseteq D_l(p)^{-1}(X).$ Prove $D_l(p)^{-1}(X) \subseteq E.$ Take $(e,e')\in D_l(p)^{-1}(X)_1$ and set $x=p(e),x'=p(e').$ Then $(x,x')\in X_1.$ Since $p$ is a $1$-covering, then there exists a unique $e''\in p^{-1}(x')$ such that $(e,e'')\in E_1.$ On the other hand $D_l(p)$ is also $1$-covering, and hence, $e''=e'.$ Therefore $(e,e')\in E_1.$ 

$(2)\Rightarrow (3).$ Lemma \ref{lemma:sub_1_covering} implies that $p$ is a $1$-covering.  Take two paths 
$(e_0,\dots,e_n)$ and $(e'_0,\dots,e'_{m})$ in $E$ such that $0\leq n,m\leq l$ and $e_0=e'_0,$ $p(e_n)=p(e_{m}').$ Set $ e:=e_0=e'_0$ and $ x:=p( e),  x':=p(e_n)=p(e_{m}').$ 
Then  $\dist(x, x')\leq \dist(e_0,e_n)\leq l$ and $(x, x')\in (D_l(X))_1.$ Since $D_l(p)$ is a $1$-covering, there exists a unique $ e'\in p^{-1}( x')$ such that $( e, e')\in (D_l(X))_1.$ Therefore $e_n= e'=e'_m.$ The second case, where $p(e_0)=p(e'_0)$ and 
$e_n=e'_{m},$ is similar. 

$(3)\Rightarrow (4).$  Chose a path $(x_0,\dots,x_n)$ in $X$ such that $x=x_0$ and $x'=x_n$ and $ \dist(x,x')\leq n\leq l.$ Since $p$ is a $1$-covering, by Lemma \ref{lemma:path_lifting} for any $e\in p^{-1}(x)$ there exists a path $(e_0,\dots,e_n)$ such that $e=e_0$ and $p(e_i)=x_i$ for any $i.$ Then we set $\alpha(e)=e_n\in p^{-1}(x').$ 
By the assumption the definition of $\alpha(e)$ does not depend on the choice of the path $(x_0,\dots,x_n).$ So, using the assumption once again, we obtain that $ \alpha : p^{-1}(x)\to p^{-1}(x')$ is a well-defined map  such that $\alpha(e)\in p^{-1}(x')$ is the unique vertex with the property $ \dist(e,\alpha(e))\leq l.$ Since we can take $n=\dist(x,x')$ in the definition of $\alpha(e),$ we see that $ \dist( e,\alpha(e))\leq \dist(x,x').$ Then \eqref{eq:dist_of_images} implies $ \dist( e,\alpha(e))= \dist(x,x').$ Dually, we can define a map $\beta: p^{-1}(x')\to p^{-1}(x)$ such that $\beta(e')\in p^{-1}(x)$ is the unique vertex such that $ \dist( \beta(e'),e')\leq l.$ It is easy to see that $\alpha \circ \beta ={\sf id}$ and $ \beta \circ \alpha ={\sf id}.$

$(4) \Rightarrow (1).$ 
Let $1\leq  n\leq l$ and $(x,x')\in D_n(X)_1$ and 
$e\in p^{-1}(x).$  
Then there exists a unique $e'=\alpha(e)\in p^{-1}(x')$ such that $(e,e')\in D_n(E)_1.$ 
Similarly, if  
$(x,x')\in D_n(X)_1$ and $e'\in p^{-1}(x'),$  then there exists a unique $e=\alpha^{-1}(e')\in p^{-1}(x)$ such that $(e,e')\in D_n(E)_1.$ Therefore $D_n(E)\to D_n(X)$ is a $1$-covering for any $1\leq n\leq l$. 
\end{proof}

\begin{corollary}\label{cor:dist=dist}
Let $p:E\to X$ be an $l$-covering. Then the following holds. 
\begin{itemize}
    \item If $e,e'\in E_0$ such that $\dist( e,e')\leq l,$ then $ \dist (p(e),p(e')) = \dist(e,e');$
    \item If $e,e'\in p^{-1}(x)$ for some $x\in X_0$ and $e\ne e',$ then $\dist( e,e')>l.$
\end{itemize}
\end{corollary}

\begin{corollary}\label{cor:l-path_lifting}
Let $p:E\to X$ be an $l$-covering of digraphs and $(x_0,\dots,x_m)$ be an accessible sequence in $X$ such that $\dist( x_i,x_{i+1} )\leq l.$ Then for any $0\leq k\leq m$ and $e\in p^{-1}(x_k)$ there exists a unique accessible sequence $(e_0,\dots,e_m)$ in $E$ such that $ \dist(e_i,e_{i+1})\leq l,$ $p(e_i)=x_i$ for any $0\leq i\leq m$ and $e_k=e.$ Moreover, in this case $\dist( e_i,e_{i+1} )= \dist(x_i,x_{i+1}).$
\end{corollary}

\subsection{Relation to covering  simplicial sets}

In this section we will use theory of covering simplicial sets \cite[Appendix I]{gabriel2012calculus}. We denote by $\Delta$ the category of non-empty finite ordinals and  by $\Delta^m$ the standard $m$-simplex $\Delta^m=\Delta(-,[m]).$ For $0\leq k\leq m$ we consider a map $i_k:[0]\to [m]$ that sends $0$ to $k.$ We use the same notation for the induced map $i_k:\Delta^0 \to \Delta^m.$ A morphism of simplicial sets $ \mathpzc{p}: \mathpzc{E}\to \mathpzc{X}$ is called covering, if for any diagram
\begin{equation}
\begin{tikzcd}
\Delta^0 \ar[d,"i_k"] \ar[r] & \mathpzc{E} \ar[d,"\mathpzc{p}"] \\
\Delta^m \ar[r] & \mathpzc{X}
\end{tikzcd}
\end{equation}
(with any $0\leq k\leq m$)
there exists a unique morphism $\Delta^m\to \mathpzc{E}$ such that the diagram is commutative. 

\begin{proposition}\label{prop:equivalent_l-covering}
The following properties of a morphism of digraphs $p:E\to X$ are equivalent.
\begin{enumerate}
\item  $p:E\to X$ is an $l$-covering.
\item $\SS^n(p):\SS^n(E)\to \SS^n(X)$ is a covering of simplicial sets for any $1\leq n\leq l.$ 
\item $\SS^l(p): \SS^l(E)\to \SS^l(X)$ is a covering of simplicial sets and $\SS^1(E)=\SS^l(p)^{-1}(\SS^1(X)).$
\end{enumerate}
\end{proposition}
\begin{proof}
$(1) \Rightarrow (2).$ It is sufficient to prove that for any natural $m\geq 1,$ any $0\leq k\leq m,$ any $(x_0,\dots,x_m)\in \SS^n(X)_m$ and any $e\in p^{-1}(x_k)$ there exists a unique $(e_0,\dots,e_m)\in \SS^n(E)_m$ such that 
$p(e_i)=x_i$ 
and $e_k=e.$ Since $\sum_{i=0}^{m-1} \dist(x_i,x_{i+1})\leq n,$ we have $\dist(x_i,x_{i+1})\leq n$ for each $0\leq i< m,$ and hence, $(x_0,\dots,x_m)$ is a path in $D_n(X).$ 
Using that $D_n(E)\to D_n(X)$ is a covering, by Lemma \ref{lemma:path_lifting} we obtain that there exists a unique path $(e_0,\dots,e_n)$ in $D_n(E)$ such that $e_k=e$ and $p(e_i)=x_i.$ 
Since $\dist(e_i,e_{i+1})\leq n,$ by Corollary \ref{cor:dist=dist} we obtain that $\dist( e_i,e_{i+1})=\dist(x_i,x_{i+1}).$ Hence $(e_0,\dots,e_m)\in \SS^n(E)_m.$ The uniqueness of the lifting to $\SS^n(E)_m$  follows from the uniqueness of the lifting to $D_n(E).$

$(2)\Rightarrow (3).$ By the assumption $D_l(X)$ is a covering. Prove that  $\SS^1(E)=(\SS^l(p))^{-1}(\SS^1(X)).$ Take $(e_0,\dots,e_m)\in \SS^l(E)_m$ such that $(p(e_0),\dots,p(e_m))\in \SS^1(X)_m.$ Since $\SS^1(p)$ is a covering and $(e_0)\in \SS^1(E)_0$, there exists a unique $(e'_0,\dots,e'_m)\in \SS^1(E)_m$ such that $e'_0=e_0$ and $p(e'_i)=p(e_i).$ Using that $\SS^l(p)$ is a covering, we obtain $(e_0,\dots,e_m) = (e'_0,\dots,e'_m).$ Therefore, $(e_0,\dots,e_m)\in \SS^1(E)_m.$ Hence $\SS^1(E)=(\SS^l(p))^{-1}(\SS^1(X)).$

$(3)\Rightarrow (1).$ Note that 
$D_n(X)_1=\SS^n(X)_1,$ 
 $D_n(E)_1=\SS^n(E)_1.$ Then the fact that $\SS^l(p)$ is a covering implies that $D_l(p)$ is a $1$-covering and the fact that $\SS^1(E)=(\SS^l(p))^{-1}( \SS^1(X))$ implies that $E=D_l(p)^{-1}(X).$ Then Proposition \ref{prop:covering_description1} implies that $p$ is an $l$-covering. 
\end{proof}

\subsection{Covering groupoids}

Here we remind some aspects of the theory of covering groupoids that can be found in \cite[Appendix I]{gabriel2012calculus},  \cite{may1999concise} and \cite{brown2006topology}. We will also give some well known results with proofs because it is difficult to find a precise reference.

Any groupoid (as well as any category) can be treated as a quiver, if we forget the composition map. We say that a morphism of groupoids $p: \mathpzc{H}\to \mathpzc{G}$ is a covering, if it is a covering  quiver. 
\begin{example}\label{ex:F(p)-covering}
If $p:E\to Q$ is a covering quiver, then the morphism between the free groupoids $F(p):F(E)\to F(Q)$ is a covering groupoid.
\end{example}
For a groupoid $\mathpzc{G}$ we denote by ${\sf Cov}(\mathpzc{G})$ the full subcategory of ${\sf Grpd}\! \downarrow\! \mathpzc{G}$ consisting of coverings of the groupoid $\mathpzc{G}$. The category of coverings of a groupoid is equivalent to the category of $\mathpzc{G}$-presheaves 
\begin{equation}\label{eq:coverings_grpd}
{\sf Cov}(\mathpzc{G})\simeq {\sf Set}^{\mathpzc{G}^{\sf op}}    
\end{equation}
\cite[Appendix I.1.2]{gabriel2012calculus}. The equivalence 
$ {\sf f} :{\sf Cov}(\mathpzc{G}) \to {\sf Set}^{\mathpzc{G}^{\sf op}}$ is defined so that ${\sf f}(p)(g) = p^{-1}(g)$ for $g\in {\sf Ob}(\mathpzc{G})$ and for a morphism $\gamma:g'\to g$ in $\mathpzc{G}$ the map ${\sf f}(p)(\gamma) : p^{-1}(g)\to p^{-1}(g')$ is defined by the unique lifting property: $h'={\sf f}(p)(\gamma)(h)$ is the only object such that there exists a morphism $\tilde \gamma: h'\to h$ with $p( \tilde \gamma)=\gamma.$ We will use the following notation 
\begin{equation}
h\gamma :=  {\sf f}(p)(\gamma)(h).
\end{equation}

For an object $g_0$ of $\mathpzc{G}$  we set $\pi_1(\mathpzc{G},g_0)=\mathpzc{G}(g_0,g_0).$ We can treat $\pi_1(\mathpzc{G},g_0)$ as a groupoid with one object $g_0.$ Then the restriction of ${\sf f}(p)$ on $\pi_1(\mathpzc{G},g_0)$ defines a right $\pi_1(\mathpzc{G},g_0)$-set ${\sf f}_{g_0}(p)=p^{-1}(g_0).$

For a group $G$ we denote by $G\text{-}{\sf Set}$ the category of right $G$-sets. Note that, if we treat a group as a groupoid with one object, then $G\text{-}{\sf Set}$ is isomorphic to the category of $G$-presheaves.

\begin{proposition}\label{prop:G-set-grpd}
If $\mathpzc{G}$ is a connected groupoid and $g_0$ is its object, then there is an equivalence of categories
\begin{equation}\label{eq:cov-grpd-eq}
{\sf Cov}( \mathpzc{G} ) \simeq \pi_1(\mathpzc{G},g_0)\text{-}{\sf Set},
\end{equation}
that sends $p$ to ${\sf f}_{g_0}(p).$ Moreover, this equivalence can be restricted to an equivalence between the category of connected covering groupoids and the category of transitive $\pi_1(\mathpzc{G},g_0)$-sets
\begin{equation}\label{eq:equivalence:cov-conn}
{\sf Cov}^{\sf conn}( \mathpzc{G} ) \simeq \pi_1(\mathpzc{G},g_0)\text{-}{\sf Set}^{\sf trans}. 
\end{equation}
\end{proposition}
\begin{proof}
Since $\pi_1(\mathpzc{G},g_0) \hookrightarrow \mathpzc{G}$ is an equivalence of categories, the restriction functor ${\sf Set}^{\mathpzc{G}^{\sf op}} \to {\sf Set}^{\pi_1(\mathpzc{G},g_0)^{\sf op}} = \pi_1(\mathpzc{G},g_0)\text{-}{\sf Set}$ is also an equivalence. The equivalence \eqref{eq:cov-grpd-eq} follows. In both categories ${\sf Cov}( \mathpzc{G} )$ and $\pi_1(\mathpzc{G}) \text{-}{\sf Set}$ coproducts are defined by disjoint unions. So, the disjoint union in ${\sf Cov}(\mathpzc{G})$ correspond the disjoint union in $\pi_1(\mathpzc{G}) \text{-}{\sf Set}.$ Connected coverings are coverings indecomposable to disjoint union of nontrivial coverings, and transitive $\pi_1(\mathpzc{G})$-sets are also indecomposable $\pi_1(\mathpzc{G})$-sets into disjoint union. Then they correspond to each other under this equivalence.
\end{proof}

A pointed groupoid is a couple $(\mathpzc{G},g_0),$ where $\mathpzc{G}$ is a groupoid and $g_0$ is its object, which is called the basepoint. By abuse of notation we set $\mathpzc{G}=(\mathpzc{G},g_0)$ and $\pi_1(\mathpzc{G})=\pi_1(\mathpzc{G},g_0).$ A morphism of pointed groupoids is a morphism of groupoids that takes the basepoint to the basepoint. The category of pointed groupoids is denoted by ${\sf Grpd}_*.$ A pointed covering is a morphism of pointed groupoids, which is a covering. The category of pointed coverings ${\sf Cov}_*(\mathpzc{G})$ is a full subcategory in the comma category ${\sf Grpd}_*\! \downarrow \! \mathpzc{G},$ whose objects are pointed coverings. The category of connected pointed coverings is denoted by ${\sf Cov}_*^{\sf conn}(\mathpzc{G}).$

\begin{lemma}\label{lemma:stab}
Let $p: \mathpzc{H}\to \mathpzc{G}$ be a pointed covering of $\mathpzc{G}.$ Then the morphism $p_*: \pi_1( \mathpzc{H})\to \pi_1(\mathpzc{G}) $ is injective and the image is equal to the stabilizer of the object $h_0$ with respect to the action of $\pi_1( \mathpzc{G})$ on $p^{-1}(g_0)$
\[ 
p_*( \pi_1( \mathpzc{H}) ) = {\sf Stab}(\pi_1( \mathpzc{G}), h_0).
\]
\end{lemma}
\begin{proof}
Since $\tilde \gamma: h_0\gamma \to h_0$ from the definition of the $\pi_1(\mathpzc{G})$-action on ${\sf f}_{g_0}$ is unique, we obtain that the map $p_*:\pi_1( \mathpzc{H}) \to \pi_1( \mathpzc{G})$ is injective. We can also note that $\tilde \gamma \in \pi_1(\mathpzc{H})$ if and only if $ h_0\gamma= h_0.$ The formula with the stabiliser follows. 
\end{proof}

For a group $G$ we denote by ${\sf Sub}(G)$ the category, whose objects are subgroups of $G$ and morphisms are embeddings. 

\begin{proposition}\label{prop:galua-grpd} Let $\mathpzc{G}$ be a pointed connected groupoid. Then there is an equivalence of categories 
\begin{equation}
{\sf Cov}_*^{\sf conn}( \mathpzc{G} ) \simeq {\sf Sub}( \pi_1( \mathpzc{G}) )
\end{equation}
that sends a pointed covering $p:\mathpzc{H}\to \mathpzc{G}$ to the image of $p_*:\pi_1(\mathpzc{H}) \to \pi_1(\mathpzc{G}).$ Moreover, for any pointed connected covering $p$ there is an isomorphism ${\sf f}_{g_0}(p)\cong \pi_1(\mathpzc{G})/p_*(\pi_1(\mathpzc{H}))$ of $\pi_1(\mathpzc{G})$-sets. 
\end{proposition}
\begin{proof}
The equivalence ${\sf f}_{g_0}$  \eqref{eq:equivalence:cov-conn} implies that there is an equivalence between the category ${\sf Cov}_*^{\sf conn}( \mathpzc{G} )$ and the category of transitive $\pi_1(\mathpzc{G})$-sets with a fixed basepoint $(\pi_1(\mathpzc{G})\text{-}{\sf Set}^{\sf trans})_*$. Objects in this category are couples $(S,s_0),$ where $S$ is a transitive $\pi_1(\mathpzc{G})$-set and $s_0\in S$ (we don't assume that the action preserves $s_0$). 
Morphisms of this category are morphisms of $\pi_1(\mathpzc{G})$-sets that preserve the basepoint. Any such a transitive $\pi_1(\mathpzc{G})$-set $S$ is isomorphic to $\pi_1(\mathpzc{G})/H,$ where $H={\sf Stab}_S(\pi_1(\mathpzc{G}),s_0).$ 
Since for any morphism $f:(S,s_0)\to (S',s_0')$ we have $f(s_0)=s'_0,$ we obtain that it can be only unique (because a morphism between transitive $G$-sets is defined by the value on one element) and it exists only if ${\sf Stab}_S(\pi_1(\mathpzc{G}),s_0)\subseteq {\sf Stab}_{S'}(\pi_1(\mathpzc{G}),s'_0).$ Therefore the functor $(\pi_1(\mathpzc{G})\text{-}{\sf Set}^{\sf trans})_* \to {\sf Sub}(\pi_1(\mathpzc{G})), (S,s_0)\mapsto {\sf Stab}_S(\pi_1(\mathpzc{G}),s_0)$ is an equivalence. Lemma \ref{lemma:stab} implies that the composite equivalence ${\sf Cov}_*^{\sf conn}(\mathpzc{G}) \to  {\sf Sub}(\pi_1(\mathpzc{G}))$ can be described as $(p:\mathpzc{H}\to \mathpzc{G} ) \mapsto p_*(\pi_1(\mathpzc{G})),$ and that ${\sf f}_{g_0}(p)\cong \pi_1(\mathpzc{G})/p_*(\pi_1(\mathpzc{H})).$
\end{proof}

\subsection{Category of {\it l}-covering digraphs}

For a simplicial set $\mathpzc{X}$ we denote by ${\sf Cov}(\mathpzc{X})$ the full subcategory of the comma category ${\sf SSet}\! \downarrow\! \mathpzc{X}$ consisting of the coverings of $\mathpzc{X}.$ The category of coverings of a simplicial set is equivalent to the category of of coverings of its fundamental groupoid ${\sf Cov}(\mathpzc{X}) \simeq {\sf Cov}( \Pi( \mathpzc{X} ) )$ 
 \cite[Appendix I.2.3]{gabriel2012calculus}. The equivalence ${\sf Cov}(\mathpzc{X}) \to {\sf Cov}( \Pi( \mathpzc{X} ) )$ is defined by taking the fundamental groupoid $\Pi.$ Therefore, for a simplicial set $\mathpzc{X}$ we have equivalences
\begin{equation} \label{eq:equivalence_coverings_ss}
{\sf Cov}(\mathpzc{X}) \simeq {\sf Cov}( \Pi( \mathpzc{X} ) ) \simeq {\sf Set}^{\Pi(\mathpzc{X})^{\sf op}}.
\end{equation}

For a digraph $X$ we denote by ${\sf Cov}_l(X)$ the full subcategory of the comma category ${\sf Digr}\! \downarrow\! X$ that consists of $l$-coverings.  

\begin{theorem} \label{theorem:coverings_equivalence}
For a digraph $X$ there are equivalences of categories
\begin{equation} 
{\sf Cov}_l(X) \simeq {\sf Cov}(\SS^l(X)) \simeq {\sf Cov}(\Pi^l(X))  \simeq {\sf Set}^{\Pi^l(X)^{\sf op}}
\end{equation} 
such that the following holds.
\begin{itemize}
\item The equivalence ${\sf Cov}_l(X) \to {\sf Cov}(\SS^l(X))$ is the functor that sends $p$ to $\SS^l(p)$ and other equivalences are described above \eqref{eq:equivalence_coverings_ss}, \eqref{eq:coverings_grpd}.
\item The composite functor 
\begin{equation*}
{\sf Fib} : {\sf Cov}_l(X)  \to {\sf Set}^{\Pi^l(X)^{\sf op}}    
\end{equation*}
is defined by ${\sf Fib}(p)(x)=p^{-1}(x)$ for any $x\in X_0;$ and for any arrow $(x',x)$ the map ${\sf Fib}(p)((x',x)):p^{-1}(x)\to p^{-1}(x')$ is defined by the unique lifting property from the definition of the $1$-covering.

\item The opposite composite functor 
\begin{equation*}
\Psi : {\sf Set}^{\Pi^l(X)^{\sf op}} \longrightarrow {\sf Cov}_l(X)    
\end{equation*}
is defined so that $\Psi(L)=(p: E\to X)$, where $E_0= \coprod_{x\in X_0} L(x);$  arrows of $E$ have the form $(L((x',x))(e),e)$ for an arrow  $(x',x)$ of $X$ and $e\in L(x).$ 
\end{itemize}
\end{theorem}
\begin{proof}
It is easy to see that the composition of the functor $(-)^l: {\sf Cov}_l(X)\to {\sf Cov}(\SS^l(X))$ with the equivalence ${\sf Cov}(\SS^l(X))\simeq {\sf Cov}(\Pi^l(X)) \simeq  {\sf Set}^{\Pi^l(X)^{\sf op}}$ is the functor ${\sf Fib}: {\sf Cov}_l(X) \to {\sf Set}^{\Pi^l(X)^{\sf op}}$ described in the statement. Therefore, it is sufficient to prove that the functor $\Psi:{\sf Set}^{\Pi^l(X)^{\sf op}} \to {\sf Cov}_l(X)$ is well defined and $\Psi \circ {\sf Fib} \cong {\sf Id}$ and $ {\sf Fib}\circ \Psi \cong  {\sf Id}.$

We need to prove that the morphism of digraphs $\Psi(L)=(p:E\to X)$ defined in the statement is an $l$-covering. Note that $p^{-1}(x)=L(x)$ and by the definition for any $(x,x')\in X_1$ we have a bijection $\alpha_{x,x'} := L((x,x'))^{-1} :p^{-1}(x)\to p^{-1}(x')$ such that $\alpha_{x,x'}(e)\in p^{-1}(x)$ is a unique vertex such that $(e,\alpha_{x,x'}(e))\in E_1.$ Then $p$ is a $1$-covering. Assume that we have two paths $(e_0,\dots,e_n)$ and $(e'_0,\dots,e'_m)$ in $E$ such that $n,m\leq l$ and $e_n=e'_m, p(e_0)=p(e'_0).$ Set $x_i=p(e_i), x_i'=p(e'_i).$ Then, using the relation \eqref{eq:rel_l}, we obtain 
\begin{equation}
\begin{split}
e_0 &= L((x_0,x_1)\cdot {\dots} \cdot (x_{n-1},x_n))(e_n) \\
& = L((x'_0,x'_1)\cdot {\dots} \cdot (x'_{m-1},x'_m))(e'_m) = e'_0.
\end{split}
\end{equation}
Similarly we prove that if we have two paths $(e_0,\dots,e_n)$ and $(e'_0,\dots,e'_m)$ in $E$ such that $n,m\leq l$ and $e_0=e'_0, p(e_n)=p(e'_m),$ then $e_n=e'_m.$ Therefore, by Proposition \ref{prop:covering_description1} $p$ is an $l$-covering. The isomorphisms ${\sf Fib} \circ \Psi \cong  {\sf Id}$ and $\Psi\circ {\sf Fib} \cong {\sf Id}$ are straightforward. 
\end{proof}

A pointed digraph is a couple $(X,x_0),$ where $X$ is a digraph and $x_0$ is its vertex.  A morphism of pointed digraphs $f:(X,x_0)\to (Y,y_0)$ is a morphism of digraphs such that $f(x_0)=y_0.$ This defines a category of pointed digraphs ${\sf Digr}_*.$ By abuse of notation we will denote $X=(X,x_0)$ and set
\begin{equation}
\pi^l_1(X) = \pi_1( \Pi^l(X),x_0 ). 
\end{equation}

\begin{corollary}\label{cor:G-sets-digr}
If $X$ is a connected pointed digraph, then there is an equivalence of categories
\begin{equation}
{\sf Cov}_l(X) \simeq \pi_1^l(X)\text{-}{\sf Set}
\end{equation}
that sends $p$ to $p^{-1}(x_0).$
\end{corollary}
\begin{proof}
It follows from Theorem \ref{theorem:coverings_equivalence} and 
Proposition \ref{prop:G-set-grpd}. 
\end{proof}

A digraph is called connected, if the underlying undirected graph is connected. Let $X$ be a pointed connected digraph. A pointed $l$-covering $p: E \to X$ is an $l$-covering, which is a morphism of pointed spaces. The category of pointed $l$-coverings ${\sf Cov}_{*,l}(X)$ is the full subcategory of ${\sf Digr}_*\!\downarrow \! X.$ The category of pointed and connected $l$-coverings will be denoted by ${\sf Cov}^{\sf conn}_{*,l}(X).$

\begin{theorem}\label{theorem:galua-digraphs}
Let $X$ be a connected pointed digraph. Then there is an equivalence of categories
\[{\sf Cov}^{\sf conn}_{*,l}(X) \simeq {\sf Sub}(\pi^l_1(X)) \]
that sends a connected pointed $l$-covering $p:E\to X$ to the image of the homomorphism  $p_*:\pi^l_1(E) \to \pi^l_1(X).$ Moreover,
for any connected pointed $l$-covering $p:E\to X$ the homomorphism $p_*:\pi^l_1(E) \to \pi^l_1(X)$ is injective and there is an isomorphism $p^{-1}(x_0)\cong \pi_1^l(X)/p_*( \pi_1^l(E) )$ of $\pi_1^l(X)$-sets.
\end{theorem} 
\begin{proof}
Theorem \ref{theorem:coverings_equivalence} implies that ${\sf Cov}^{\sf conn}_{*,l}(X)$ is equivalent to the category pointed connected covering groupoids ${\sf Cov}^{\sf conn}_{*}(\Pi^l(X)).$ Then the assertion follows from Proposition \ref{prop:galua-grpd} and Lemma \ref{lemma:stab}.
\end{proof}

\subsection{Universal {\it l}-covering and Deck transformations}
Let $X$ be a connected digraph. 
An $l$-covering 
$p_U:U\to X$ is 
called universal, if 
$\Pi^l(U)$ is indiscrete (i.e. all hom-sets are one-element). In other words, $p_U:U\to X$ is universal if $U$ is connected and $\pi_1(U)$ is trivial (since $U$ is connected $\pi_1(U)$ does not depend on the choice of the basepoint in $U$ up to isomorphism). By Theorem 
\ref{theorem:galua-digraphs} the universal covering exists, and if we fix basepoints of
$X$ and $U,$ then 
$p_U:U\to X$ is the initial 
object of ${\sf Cov}^{\sf conn}_{*,l}(X).$ In particular, it is unique. Moreover, if we fix a basepoint it is unique up to unique isomorphism, but if we don't fix a basepoint, the universal covering has non-trivial automorphisms which we are going to describe. In the case of spaces such automorphisms are called Deck transformations.

If $p_U:U\to X$ is a universal $l$-cover, for two vertices $u,u'$ of $U$ we denote by $\alpha_{u,u'}$ the unique morphism in $\Pi^l(U)(u,u').$ 

\begin{proposition}\label{prop:fibers_of_universal} 
Let $X$ be a connected digraph and let $p_U:U\to X$ be the universal $l$-covering with some chosen basepoints $u_0$ in $U$ and $x_0=p(u_0)$ in $X.$ Then there is a bijection  
\begin{equation}
\theta : U_0 \cong \Pi^l(X)(X_0,x_0)  \end{equation}
defined by the formulas 
\begin{equation}\label{eq:theta}
\theta(u)=p(\alpha_{u,u_0}), \hspace{1cm} \theta^{-1}(\gamma) =  u_0\gamma.
\end{equation}
Moreover, $(u,u')$ is an arrow in $U$ if and only if $(p(u),p(u'))$ is an arrow in $X$ and  $(p(u),p(u'))\cdot\theta(u')  =\theta(u).$
\end{proposition}
\begin{proof} 
Since $\Pi^l(U)$ is indiscrete, the map $p^{-1}(x)\to \Pi^l(U)(p^{-1}(x),u_0)$ sending $u$ to $\alpha_{u,u_0}$ is a bijection. Since 
$\Pi^l(p):\Pi^l(U)\to \Pi^l(X)$ 
is a covering of groupoids, $p$ induces a bijection $\Pi^l(U)(p^{-1}(x),u_0) \cong \Pi^l(X)(x,x_0).$ Therefore, we obtain that $\theta$ gives a bijection $p^{-1}(x)\cong \Pi^l(X)(x,x_0).$ Taking the union by all $x,$ we obtain the bijection $U_0\cong \Pi^l(X)(X_0,x_0).$
The description of arrows in $U$ follows from the definition of $1$-cover. 
\end{proof}

\begin{remark}
Proposition \ref{prop:fibers_of_universal} can be considered as a recipe for explicit construction of the universal $l$-cover.   Define $U$ so that $U_0=\Pi^l(X)(X_0,x_0),$ and arrows of $U$ have the form $((x,x')\cdot \gamma,\gamma)$ for $\gamma\in U_0$ and an arrow $(x,x').$ The map $p_U$ is defined by the formula $p_U(\gamma:x\to x_0) = x.$  
\end{remark}

\begin{proposition}[Deck transformations] \label{prop:Deck}
Let $X$ be a connected digraph and let $p_U:U\to X$ be the universal $l$-covering with chosen basepoints $u_0$ in $U$ and  $x_0=p(u_0)$ in $X$. Then there exists a left action of $\pi_1^l(X)$ on $U$ (which is a right $(\pi_1^l(X),\cdot)$-action) by automorphisms of covers defined by the formula
\begin{equation}\label{eq:Deck}
\gamma u  := \theta^{-1}(\gamma\theta(u)). 
\end{equation}
This action induces an isomorphism
\begin{equation}
\pi^l_1(X) \cong {\sf Aut}_{{\sf Cov}_l(X)}(p_U).
\end{equation}
\end{proposition}
\begin{proof}
It is easy to check that action on $U$ defined so that $\theta$ is an isomorphism of $\pi_1^l(X)$-sets is an action of $\pi_1^l(X)$ by automorphisms of covers. The action on the fibre $p^{-1}_U(x_0)$ is isomorphic to the action of $\pi_1^l(X)$ on itself by shifts. For any group $G$ there is an isomorphisms $G \cong  {\sf Aut}_{G\text{-}{\sf Set}}(G)$ such that $g\mapsto (x\mapsto gx).$ Combining this with Corollary \ref{cor:G-sets-digr} we obtain that the action induces an isomorphism  $\pi_1^l(X) \cong {\sf Aut}_{{\sf Cov}_l(X)}(p_U).$
\end{proof}

\subsection{Homotopy lifting property for 2-covering digraphs} 
For two morphisms of digraphs $f,g:X\to Y$ we define the distance $\dist(f,g)$ as the supremum of the distances $\dist(f(x),g(x)), x\in X_0.$ We say that two morphisms of digraphs $f,g:X\to Y$ are one-step homotopic, if either $\dist( f, g )\leq 1,$ or $\dist( g, f)\leq 1.$ A sequence of morphisms $(f_0,\dots,f_n), f_i:X\to Y$ is called a homotopy, if $f_i$ is one-step homotopic to $f_{i+1}.$ Two morphisms $f,g: X\to Y$ are homotopic, if there is a homotopy $(f_0,\dots,f_n)$ such that $f_0=f,f_n=g.$ It is easy to see that a homotopy $h=(f_0,\dots,f_n)$ can be also defined as a morphism 
\begin{equation}
h : X \square I_n \longrightarrow Y,
\end{equation}
where $I_n$ is a digraph with the vertex set $\{0,\dots,n\}$ and one arrow between $i$ and $i+1$ with some choice of directions. 
\[ I_n : \hspace{1cm} 0 \to 1 \leftarrow 2 \leftarrow 3 \to 4 \leftarrow \dots \to n \]

\begin{proposition} Let $p:E\to Y$ be a $2$-covering digraph and let $(f_0,\dots,f_n), f_i:X\to Y$ be a homotopy. Then for any morphism $g_0:X\to E$ such that $p g_0=f_0$ there exists a unique homotopy $(g_0,\dots,g_n)$ such that $p g_i=f_i.$    
\end{proposition}
\begin{proof}
It is easy to see that it is sufficient to prove this for the case $n=1.$ Assume that $f_0,f_1:X\to Y$ are one-step homotopic and $g_0:X\to E$ is a morphism such that $p g_0=f_0.$ So, either $\dist(f_0,f_1)\leq 1,$  or $\dist(f_1,f_0)\leq 1.$  Further we assume that $\dist(f_0,f_1)\leq 1,$ the second case is similar. Using that $p$ is a $1$-covering we obtain that for each $x$ there a unique vertex $g_1(x)\in p^{-1}( f_1(x) )$ such that $\dist( g_0(x),g_1(x))= \dist(f_0(x),f_1(x))\leq 1.$ 

Let us prove that $g_1$ is a morphism $X\to E.$ Assume that $\dist(x,x')= 1.$ Then there exists a unique $e'\in p^{-1}(f_1(x'))$ such that $ \dist(g_1(x),e') =\dist( f_1(x),f_1(x') )\leq 1.$ Then $(g_0(x),g_0(x'),g_1(x'))$ and $(g_0(x),g_1(x),e')$ are $2$-paths such $p(g_1(x'))=p(e')$. By Proposition  \ref{prop:covering_description1} we obtain $g_1(x')=e'$ (here we use that $p$ is a $2$-covering). Therefore, $ \dist(g_1(x),g_1(x')) = \dist(f_1(x),f_1(x'))\leq 1.$ Hence $g_1$ is a morphism.

By the construction it is obvious that $g_1:X\to E$ is the unique morphism such that $\dist(g_0,g_1) \leq 1$  and $pg_1=f_1.$ We need to check that, if $g'_1:X\to E$ is a morphism such that $p g'_1=f_1$ and $\dist(g_1',g_0 )\leq 1,$ then $g_1=g'_1.$ Indeed, using the triangle inequality and the inequalities $\dist(g'_1(x),g_0(x))\leq 1$ and $\dist(g_0(x),g_1(x))\leq 1,$  we obtain $\dist(g'_1(x),g_1(x))\leq \dist(g'_1(x),g_0(x)) +  \dist(g_0(x),g_1(x))\leq 2.$ Since $pg_1=f_1=pg'_1,$ we have  $p(g_1(x))=p(g'_1(x)).$ By Corollary \ref{cor:dist=dist} we obtain $g_1(x)=g'_1(x).$ 
\end{proof}

\subsection{Path chains of 2-covering digraphs}

Recall that for a digraph $X$ and two its vertices $x,x'$ we denote by $\Omega_n(X,x,x')$ the set of path chains which are linear combinations of paths with tail in $x$ and head in $x'.$ More generally, for a subset in the set of vertices $W,W'\subseteq X_0,$ we denote by $\Omega_n(X,W,W')\subseteq \Omega_n(X)$ the set of path chains which are linear combinations of paths with tail in a vertex of $W$ and head in a vertex of $W'.$ The formula \eqref{eq:clusters} implies that
\begin{equation}
\Omega_n(X,W,W') = \bigoplus_{(x,x')\in W\times W'} \Omega_n(X,x,x').    
\end{equation}
Mainly, we will use the case when one of the sets is one-element $W=\{x\}$ or $W'=\{x'\}$ and use the notation $\Omega_n(X,x,W')=\Omega_n(X,\{x\},W')$ and $\Omega_n(X,W,x')=\Omega_n(X,W,\{x'\})$ in this case. Similarly we define $\AA_n(X,W,W')\subseteq \AA_n(X)$ and $\RR_n(X,W,W')\subseteq \RR_n(X).$

\begin{proposition}\label{prop:omega-2-coverings}
Let $p:E\to X$ be a 2-covering, $x$ be a vertex of $X$ and $e$ be a vertex of $E.$ Then $p$ induces isomorphisms 
\begin{equation}\label{eq:iso_omega_1}
\Omega_n(E,e,p^{-1}(x)) \cong \Omega_n(X,p(e),x)   
\end{equation}
and  
\begin{equation}
\Omega_n( E, p^{-1}(x),e ) \cong \Omega_n(X,x,p(e)).
\end{equation}
\end{proposition}
\begin{proof} First we introduce some notations and make some notes for arbitrary digraph $Y.$
For sets of vertices $W,W'\subseteq Y_0$ we denote by $RP_n(Y,W,W')$ the set of regular paths $(y_0,\dots,y_n)$ (i.e. $\dist(y_i,y_{i+1})=1$) such that $y_0\in W$ and $y_n\in W'.$ For a number $1\leq k\leq n$ we denote by $RN_n^{k}(Y,W,W')$ the set of sequences $(y_0,\dots,y_n)$ such that $y\in W,y'\in W'$  and $\dist(y_i,y_{i+1})=1$ for $i\neq k$ and $\dist(y_k,y_{k+1})=2.$ Note that $\AA_n(Y,W,W')$ is a submodule in $\RR_n(Y,W,W')$ generated by $RP_n(Y,W,W').$ 
We denote by $B_n^k(Y,W,W')$ the submodule of $\RR_n(Y,W,W')$ generated by $RN^k_n(Y,W,W').$
Note that for any $1\leq k \leq n-1$ (here we assume $k\ne 0,n$) we have $d_k(\AA_n(Y,W,W'))\subseteq \AA_{n-1}(Y,W,W')\oplus B_{n-1}^k(Y,W,W').$ We denote by 
\begin{equation}
 d_k':\AA_n(Y,W,W') \to B_{n-1}^k(Y,W,W')   
\end{equation}
the composition of $d_k$ with the projection to $B_{n-1}^k(Y,W,W')$. Then by Proposition \ref{prop:description_of_PC} we have
\begin{equation}\label{eq:Omega_formula}
\Omega_n(Y,W,W') = \{a\in \AA_n(Y,W,W')\mid   d_k'(a)=0 \text{ for } 1\leq k\leq n-1\}.  
\end{equation}

Now  assume that $p:E\to X$ is a 2-covering, $x$ is a vertex of $X$ and $e$ is a vertex of $E.$  Lemma \ref{lemma:path_lifting} and Corollary \ref{cor:l-path_lifting} imply that 
$p$ induces bijections
\begin{equation}
RP_n(E,e,p^{-1}(x)) \cong  RP_n(X,p(e),x),
\end{equation}
\begin{equation}
RN^k_n(E,e,p^{-1}(x)) \cong  RN^k_n(X,p(e),x).
\end{equation}
Therefore we obtain that the horizontal arrows of the following commutative diagram are isomorphisms
\begin{equation}
\begin{tikzcd}
\AA_n(E,e,p^{-1}(x)) \ar[rr,"\cong"] \ar[d,"d'_k"] && \AA_n(X,p(e),x) \ar[d,"d'_k"] \\ 
B^k_{n-1}(E,e,p^{-1}(x)) \ar[rr,"\cong"] &&  B^k_{n-1}(X,p(e),x).
\end{tikzcd}
\end{equation}
The first isomorphism of the statement follows from this diagram and formula \eqref{eq:Omega_formula}. The second isomorphism is similar. 
\end{proof}

\begin{proposition}\label{prop:cover:clusters}
Let $p:E\to X$ be a 2-covering and $\gamma$ be a morphism in $\Pi^2(E).$  Then $p$ induces an isomorphism of modules of clusters
\begin{equation}
\Omega_n(E,\gamma) \cong \Omega_n(X,p(\gamma)).     
\end{equation}
\end{proposition}
\begin{proof} Combining 
Proposition \ref{prop:omega-2-coverings} and Proposition \ref{prop:clusters} for any vertices $x,x'$ of $X$ and $e\in p^{-1}(x)$ we obtain that $p$ induces an isomorphism 
\begin{equation}
\bigoplus_{\gamma \in \Pi^2(E)(e,p^{-1}(x'))} \Omega_n(E,\gamma) \cong \bigoplus_{\beta \in \Pi^2(X)(x,x')} \Omega_n(X,\beta).
\end{equation}
Since $\Pi^2(E)\to \Pi^2(X)$ is a covering of groupoids, $\Pi^2(E)(e,p^{-1}(x'))\to \Pi^2(X)(x,x')$ is a bijection. Therefore, the summands are in one-to-one correspondence to each other. Using that the image of $\Omega_n(E,\gamma)$ lies in $\Omega_n(X,p(\gamma)),$ we obtain the assertion.
\end{proof}

\subsection{Spectral sequence of the universal 2-cover}

Let $X$ be a connected digraph and $p_U:U\to X$ be the universal $2$-covering. Since for any fixed basepoint we have an isomorphism (Proposition \ref{prop:Deck})  
\[\pi_1^2(X)\cong {\sf Aut}_{{\sf Cov}_2(X)} (p_U),\] 
we can identify $\pi^2_1(X)$ with the group ${\sf Aut}_{{\sf Cov}_2(X)} (p_U).$ This is convenient, because this definition is independent of the choice of a basepoint. The action of $\pi^2_1(X)$ on $U$ induces an action of $\pi^2_1(X)$ on $\Omega_n(U).$ This defines a structure of $\KK[\pi^2_1(X)]$-module on $\Omega_n(U).$

\begin{proposition}\label{prop:omega_U}
There is an isomorphism of left $\KK[\pi_1^2(X)]$-modules
\begin{equation}
\Omega_n(U)\cong \KK[\pi_1^2(X)]\otimes_{\KK} \Omega_n(X),
\end{equation}
which makes the diagram 
\begin{equation}\label{eq:commutative_triangle}
\begin{tikzcd}
\Omega_n(U) \ar[rr,"\cong"] \ar[rd,"\Omega_n(p)"'] & & 
\KK[\pi_1^2(X)] \otimes \Omega_n(X) \ar[ld,"\varepsilon\otimes {\sf id}"] 
\\
& \Omega_n(X) &
\end{tikzcd}
\end{equation}
commutative, where $\varepsilon:\KK[\pi_1^2(X)] \to \KK$ is the augmentation of the group algebra. 
\end{proposition}
\begin{proof}
It is easy to see that for any vertices $x,x'$ the module $\Omega_n(U,p^{-1}(x),p^{-1}(x'))$ is a $\KK[\pi^2_1(X)]$-submodule of $\Omega_n(U).$ Moreover, there is an isomorphism of $\KK[\pi_1^2(X)]$-modules
\begin{equation}
\Omega_n(U)\cong \bigoplus_{x,x'} \Omega_n(U,p^{-1}(x),p^{-1}(x')).
\end{equation}
Therefore, it is sufficient to prove the isomorphism 
\begin{equation}
\Omega_n(U,p^{-1}(x),p^{-1}(x')) \cong \KK[\pi^2_1(X)] \otimes  \Omega_n(X,x,x')    
\end{equation}
for any vertices $x,x'.$  Further in the proof we fix a point $x',$ and consider this arbitrary point as the basepoint $x_0:=x'$ of $X$ and chose a basepoint $u_0\in p^{-1}(x_0).$ Hence, it is sufficient to prove the isomorphism 
\begin{equation}
\Omega_n(U,p^{-1}(x),p^{-1}(x_0)) \cong \KK[\pi^2_1(X)] \otimes  \Omega_n(X,x,x_0)    
\end{equation}
(We consider different points as a basepoint to define the isomorphism on different summands. That is why it is convenient for us to identify the group $\pi^2_1(X)$ with the group 
${\sf Aut}_{{\sf Cov}_2(X)}(U),$ which is independent of the choice of the basepoint). Recall that for any $u,u'\in U_0$ we denote by $\alpha_{u,u'}:u\to u'$ the unique morphism in $\Pi^2(U).$ Then $\Omega_n(U,u,u')\cong \Omega_n(U,\alpha_{u,u'}).$
We define a map 
\begin{equation}
f: \Omega_n(U,p^{-1}(x),p^{-1}(x_0)) \to \KK[\pi^2_1(X)] \otimes \Omega_n(X,x,x_0)    
\end{equation}
such that its restriction $f_{u,u'}:\Omega_n(U,u,u')\to \KK[\pi^2_1(X)] \otimes \Omega_n(X,x,x_0)$ is defined by the formula
\[
f_{u,u'}(\omega)=\theta(u') \otimes p(\omega)
\]
(for the definition of $\theta$ see \eqref{eq:theta}).
By Proposition \ref{prop:cover:clusters} $f_{u,u'}$ induces an isomorphism $\KK$-modules 
\begin{equation}
   \Omega_n(U,u,u') \cong  (\KK\cdot \theta(u')) \otimes \Omega_n(X,p(\alpha_{u,u'})). 
\end{equation}
If we take the direct sum by all $u\in p^{-1}(x_0)$ and $u'\in p^{-1}(x'),$ we obtain that $f$ is an isomorphism of $\KK$-modules. Now, in order to prove the isomorphism of $\KK[\pi_1^2(X)]$-modules, we only need to check that $f$ is a homomorphism of $\KK[\pi_1^2(X)]$-modules. Take $\gamma \in \pi_1^2(X).$ Since Deck transformations are automorphisms of covers, we obtain  $p(\gamma  \omega)=p(\omega).$ Using this and the definition of the Deck transformation \eqref{eq:Deck} we get  
\begin{equation}
f(\gamma \omega) 
= f_{\gamma  u, \gamma u'}(\gamma \omega) 
= 
\theta(\gamma  u') 
\otimes p(\omega) 
= \gamma \theta(u')  \otimes p(\omega) = \gamma  f(\omega). 
\end{equation}
The fact that  the triangle \eqref{eq:commutative_triangle} is commutative is obvious: $(\varepsilon\otimes {\sf id})(f_{u,u'}(\omega))= (\varepsilon\otimes {\sf id})(\theta(u')\otimes p(\omega))=p(\omega).$
\end{proof}

\begin{corollary}\label{cor:coinvariants_of_omega_u}
There is an isomorphism of complexes
\begin{equation}
\KK \otimes_{\KK[\pi_1^2(X)]}\Omega(U)\cong \Omega(X)
\end{equation}
defined by the formula $a\otimes \omega\mapsto ap(\omega).$
\end{corollary}
\begin{proof}
    Using that $p(\gamma \omega)=p(\omega)$ for any $\gamma\in \pi_1^2(X),$ we see that there is a well-defined morphism of complexes $\KK \otimes_{\KK[\pi_1^2(X)]}\Omega(U)\to \Omega(X)$ defined by $a\otimes \omega\mapsto ap(\omega).$  Then we need to check that for any $n$ the map $\KK \otimes_{\KK[\pi_1^2(X)]}\Omega_n(U)\to \Omega_n(X)$ is an isomorphism. It follows from the commutativity of the diagram \eqref{eq:commutative_triangle} and the isomorphism $\KK\otimes_{\KK[\pi_1^2(X)]}\KK[\pi_1^2(X)]\cong \KK$.
\end{proof}

\begin{remark}
Note that we don't claim that the complexes $\Omega(U)$ and $\KK[\pi_1^2(X)]\otimes \Omega(X)$ are isomorphic. 
\end{remark}

\begin{theorem}\label{theorem:covering_spectral}
Let $X$ be a connected digraph and 
$p_U:U\to X$ be the universal $2$-covering.Then there is a spectral sequence $E$ of homological type such that  
\begin{equation}
    E^2_{i,j} = H_i(\pi^2_1(X),\PH_j(U)) \Rightarrow \PH_{i+j}(X),
\end{equation}
where $\PH_*(U)$ is treated as a $\pi^2_1(X)$-module with the action by Deck transformations, and $H_i(\pi_1^2(X),-)$ is the group homology of the group $\pi_1^2(X).$
\end{theorem}
\begin{proof}
We take the complex $\Omega_n(U),$ the functor of $\pi_1^2(X)$-coinvarints $H_0(\pi_1^2(X), -  ):{\sf Mod}(\KK[\pi_1^2(X)]) \to {\sf Mod}(\KK)$ and consider two associated hyperhomology spectral sequences ${}^{I}E,{}^{I\!I}E$ (\cite[Prop.5.7.6]{weibel1995introduction}). 
\begin{equation}
\begin{split}
{}^{I}E_{i,j}^1 = H_j(\pi_1^2(X), \Omega(U)_j)) &\Rightarrow \mathbb{H}_{i+j}( \pi_1^2(X), \Omega(U)).  \\ 
{}^{I\!I}E_{i,j}^2 = H_i( \pi_1^2(X),\PH_j(U)) &\Rightarrow  \mathbb{H}_{i+j}( \pi_1^2(X), \Omega(U)).
\end{split}
\end{equation}
If $S$ is a set and $\KK\cdot  S$ is a free module generated by $S,$ then for any $\KK$-module $M$ we have isomorphisms $(\KK\cdot S) \otimes_\KK M \cong M^{\oplus S} \cong (\ZZ \cdot S) \otimes_{\ZZ} M.$ Therefore, 
 the isomorphism $\KK[\pi_1^2(X)]\cong \KK\cdot \pi_1^2(X)$ and Proposition \ref{prop:omega_U} imply 
\begin{equation}
\Omega_n(U)\cong \KK[\pi_1^2(X)]\otimes_\KK \Omega_n(X)\cong \mathbb{Z}[\pi_1^2(X)]\otimes_{\mathbb{Z}} \Omega_n(X).    
\end{equation}
Combining this with \cite[Ch.III, Cor.6.6]{brown2012cohomology} we obtain      $H_j(\pi_1^2(X),\Omega(U))=0$ for $j\neq 0$. Corollary \ref{cor:coinvariants_of_omega_u} implies $H_0(\pi_1^2(X),\Omega(U))=\Omega(X).$ Therefore, ${}^IE^1_{i,j}=0$ for $j\ne 0$ and 
$ {}^IE^1_{*,0} = \Omega(X).$ It follows that $\mathbb{H}_*( \pi_1^2(X), \Omega(U))= \PH_*(X).$ Then ${}^{I\!I}E$ is the required spectral sequence. 
\end{proof}

\begin{corollary}\label{cor:uni_cov}
Under assumption of Theorem \ref{theorem:covering_spectral}, if $\widetilde{\PH}_*(U)=0 ,$ then 
\begin{equation}
\PH_*(X)\cong H_*(\pi_1^2(X)).
\end{equation}
\end{corollary}

\subsection{Examples}

It is easy to check that the morphism of digraphs illustrated on the following picture is a $2$-covering but not a $3$-covering.

\ 

\vbox{
\[ E: \ 
\begin{tikzcd}[row sep=5mm, column sep=5mm]
\dots \ar[r] & -2 \ar[r] & -1 \ar[r] & 0 \ar[r] & 1 \ar[r] & 2 \ar[r] & 3 \ar[r] &  \dots    
\end{tikzcd} 
\]

\[\ \hspace{3mm} \Big\downarrow \  \  p\]

\[
X: \ 
\begin{tikzcd}[row sep=3mm, column sep=3mm]
& 1 \ar[dr]  & \\ 
0\ar[ur] & & 2 \ar[ll]
\end{tikzcd} \hspace{5mm} \ \]
}
\noindent Here $p(x) = x ({\rm mod}\ 3).$ One can generalize this example considering a cycle of length $n\geq 3.$ This will be an example of an $n-1$-covering but not an $n$-covering. In this case it is easy to see that $\widetilde{\PH}_*(E)=0,$ $\pi_1^2(E)=1$ and $\pi_1^2(X)=\ZZ.$ Therefore Corollary \ref{cor:uni_cov} implies that  $\PH_*(X)=H_*(\ZZ).$ 

In the next example we consider a directed triangle instead of the cyclic triangle from the previous example. It easy to check that morphism of digraphs illustrated on the following picture is a 1-covering but not a 2-covering.

\ 

\vbox{
\[ E: \ 
\begin{tikzcd}[row sep=5mm, column sep=5mm]
\dots \ar[r]  & -2 \ar[r] & -1   & 0 \ar[l] \ar[r]   & 1 \ar[r]  & 2   & 3 \ar[l] \ar[r]   &  \dots     
\end{tikzcd} 
\]

\[\ \hspace{4mm} \Big\downarrow \  \  p\]

\[
X: \ 
\begin{tikzcd}[row sep=3mm, column sep=3mm]
& 1 \ar[dr]  & \\ 
0\ar[ur] \ar[rr] & & 2 
\end{tikzcd} \hspace{5mm} \ \]
}
\noindent The fact that it is not a 2-covering can be checked by hand but it also obviously follows from the fact that $\pi^2_1(X)=1$ in this case, and hence, $X$ has no non-trivial $2$-coverings.

A more complicated example of a universal 2-covering is illustrated on the following picture.

\ 

\vbox{
\[ U: 
\begin{tikzcd}[row sep=5mm, column sep=5mm]
\cdots \ar[r]\ar{rd}  
& x_{-2} \ar[r]\ar{rd}
& x_{-1}\ar[r]\ar{rd}
& x_{0} \ar[r]\ar{rd}
& x_{1} \ar[r]\ar{rd}
& x_{2} \ar[r]\ar{rd}
& x_{3} \ar[r]\ar{rd}
& \cdots \\
\cdots \ar{r}\ar{ru}
& y_{-2}\ar{r}\ar{ru}
& y_{-1}\ar{r}\ar{ru}
& y_{0}\ar{r}\ar{ru}
& y_{1}\ar{r}\ar{ru}
& y_{2}\ar{r}\ar{ru}
& y_{3}\ar{r}\ar{ru}
& \cdots 
\end{tikzcd} 
\]  

\[\ \hspace{10mm} \Big\downarrow \  \  p_U\]

\[ X: 
\begin{tikzcd}
x_0 \ar{rrd} \ar[rrr]  &&& 
x_1 \ar[ddd] \ar{ddl}
\\
& y_0 
\ar{r} 
\ar{rru} 
& 
y_1 \ar{d} 
\ar{ddr} \\
&
y_3
\ar{u} 
\ar{uul}
& 
y_2 
\ar{l} 
\ar{lld} 
\\
x_3  
\ar{uur} \ar[uuu] 
&&& 
x_2 \ar{llu} \ar[lll]
\end{tikzcd}
\]
}
\noindent In this case one can also compute that $\widetilde{\PH}_*(U)=0,$  $\pi_1^2(U)=1$ and $\pi_1^2(X)={\sf Aut}(p_U)=\ZZ.$ Therefore $\PH_*(X)=H_*(\ZZ).$

\section{\bf Cayley digraphs}
\subsection{1-fundamental group of a Cayley digraph} 
Let $G$ be a group and $S\subseteq G\setminus \{1\}$ be a generating subset. Then the Cayley digraph $\Cay(G,S)$ is a digraph, whose set of vertices is $G$ and the set of arrows consists of couples $(g,gs),$ where $g\in G$ and $s\in S.$ The assumption that $S$ generates $G$ implies that the digraph is connected. There is also a natural choice of a basepoint: $1\in G.$ So $\Cay(G,S)$ is a connected pointed digraph. 

We denote by $F(S)$ the free group generated by $S,$ by  $\iota:S\to F(S)$ the canonical embedding and we set $\bar s=\iota(s).$ We use these notations in order to distinguish $s\in G$ from $\bar s\in F(S).$ Consider the epimorphism $\sigma : F(S)\epi G$ that sends $\bar s$ to $s$ and set
\begin{equation}
\Rel(G,S) = {\sf Ker}(\sigma : F(S)\epi G).
\end{equation}

We denote by $Q_S$ the quiver with one vertex and $(Q_S)_1=S.$ The free groupoid defined by $Q_S$ is the free group generated by $S$ (which is treated as a groupoid with one object) $F(Q_S)=F(S).$ Consider a morphism of quivers 
\begin{equation}
 r :Q(\Cay(G,S))\to Q_S    
\end{equation}
that sends an arrow $(g,gs)$ to $s,$ for $s\in S.$ 

\begin{proposition} 
The morphism of free groupoids $F(r) : F(\Cay(G,S))\to F(S)$ is a covering  groupoid corresponding to the subgroup 
$\Rel(G,S)$ of $F(S).$ In particular, it induces an isomorphism 
\begin{equation}\label{eq:iso_pi^1-Cay}
r_* :\pi_1^1(\Cay(G,S)) \cong \Rel(G,S).
\end{equation}
\end{proposition}
\begin{proof} 
It is easy to see that $r$ is a covering quiver. Example \ref{ex:F(p)-covering} shows that $F(r)$ is a covering groupoid. Note that $F(\Cay(G,S))=\Pi^1(\Cay(G,S)).$ By Proposition \ref{prop:galua-grpd} the map $\pi_1^1(\Cay(G,S)) \to F(S)$ is injective.  So we only need to prove that the image of $\pi_1^1(\Cay(G,S))\to F(S)$ is $\Rel(G,S).$

We claim that for any morphism $\gamma$ in $ F(Q(\Cay(G,S)))$ there is an equation $\sigma(F(r)(\gamma)) ={\sf dom}(\gamma)^{-1}  {\sf cod}(\gamma)$ in $G.$ Indeed, these two maps $\sigma F(r)$ and ${\sf dom}(-)^{-1} {\sf cod}(-)$ are both morphisms of groupoids and they coincide on the generating set $(g,gs).$ Therefore, ${\sf dom}(\gamma)={\sf cod}(\gamma)$ if and only if $F(r)(\gamma)\in \Rel(G,S).$ 
\end{proof}

\subsection{\textit{l}-fundamental groupoid of a Cayley digraph}

We denote by $W^l(S)\subseteq F(S)$ the set of all elements of the free group $F(S)$ that can be presented as $(\bar s_1 \dots  \bar s_n) \cdot (\bar  t_1\dots \bar t_m)^{-1},$ where $n,m\leq l$ and $s_i,t_j\in S$. Consider a normal subgroup $\Rel^l(G,S)$ of $F(S)$ generated by the intersection $\Rel(G,S)\cap W^l(S).$
\begin{equation}
\Rel^l(G,S) = \langle \Rel(G,S) \cap W^l(S) \rangle^{F(S)}. 
\end{equation}
We also set
\begin{equation}
\rho^l(G,S) = \Rel(G,S)/ \Rel^l(G,S).    
\end{equation}

Further we treat  $\Pi^l(X)$ as a quotient of $\Pi^1(X)$ by relations \eqref{eq:rel_l}. In particular, $\pi_1^l(X)$ is a quotient of $\pi_1^1(X).$

\begin{proposition} \label{prop:pi^l-Cay}
The isomorphism \eqref{eq:iso_pi^1-Cay} induces an isomorphism 
\begin{equation}
\pi^l_1(\Cay(G,S)) \cong \rho^l(G,S).
\end{equation}
\end{proposition}
\begin{proof} For simplicity in the proof we omit $\Cay(G,S)$ and $(G,S)$ in notations: $\Pi^1=\Pi^1(\Cay(G,S)), \Rel=\Rel(G,S)$ and so on. For a sequence of elements $s_1,\dots s_n\in S$ and $g\in G$ we consider the morphism 
\[
w_g(s_1,\dots s_n)=(g,gs_1)\cdot (gs_1,gs_1s_2) \dots (gs_1\dots s_{n-1},gs_1\dots s_n)\] 
in $\Pi^1(g,gs_1\dots s_n)$. Denote by $H$ the normal subgroup of $\pi_1^1$ generated by elements of the form $w_1(s_1,\dots, s_n)\cdot w_1(t_1,\dots,t_m)^{-1},$ where $s_1\dots s_n = t_1 \dots t_m $ and $n,m\leq l.$ We just need to prove that $H$ is the kernel of the map $\pi_1^1\to \pi_1^l$ and 
$r_*(H)=\Rel^l.$ The fact that the image of $H$ in $\Rel$ is $\Rel^l$ is straightforward: the image of $w_1(s_1,\dots, s_n)\cdot w_1(t_1,\dots,t_m)^{-1}$ is $(\bar s_1 \dots \bar s_n)\cdot (\bar t_1\dots \bar t_m)^{-1}\in W^l(S).$ Prove that $H$ is the kernel of $\pi_1^1\to \pi_1^l.$ 

By Proposition \ref{prop:pi^l} the groupoid $\Pi^l$ is the quotient of $\Pi^1$ by the minimal congruence relation such that $w_g(s_1,\dots,s_n)\sim w_g(t_1,\dots,t_m)$
where $s_1 \dots s_n=t_1\dots t_m,$ $s_i,t_i\in S,$  $g\in G$ and $n,m\leq l.$ Note that the morphism $\Pi^1\to \Pi^l$ is identical on objects, and hence $\Pi^l=\Pi^1/\mathpzc{H},$ where $\mathpzc{H}$ is a totally disconnected normal subgroupoid \cite[(8.3.2)]{brown2006topology}. The totally disconnected normal subgroupoid $\mathpzc{H}$ is generated by the elements of the form $w_g(s_1,\dots,s_n) \cdot w_g(t_1,\dots,t_m)^{-1}.$ Then we need to prove that $H=\mathpzc{H}(1,1).$ The group $ \mathpzc{H}(1,1)$ is the subgroup of $\pi_1^1$ generated by elements of the form
\begin{equation}\label{eq:the_element}
\gamma \cdot  w_g(s_1,\dots,s_n) \cdot w_g(t_1,\dots,t_m)^{-1} \cdot \gamma^{-1},    
\end{equation}
where $\gamma\in \Pi^1(1,g),$  $s_1 \dots s_n = t_1 \dots t_m$ and $n,m\leq l.$ So we just need to prove that an element of the form \eqref{eq:the_element} is in $H.$ Since $r_*: \pi_1^1 \to \Rel$ is an isomorphism and $r_*(H)=\Rel^l$, it is enough to check that the image of the element \eqref{eq:the_element} in $\Rel$ is in $\Rel^l.$ The image is equal to $F(r)(\gamma)\cdot (s_1\dots s_n)\cdot (t_1\dots t_m)^{-1} \cdot F(r)(\gamma)^{-1},$ which is conjugated to the element $(s_1\dots s_n)\cdot (t_1\dots t_m)^{-1}\in \Rel^l.$
\end{proof}

\subsection{Universal \textit{l}-covering of a Cayley digraph} 

For a subset $S\subseteq G$ we set 
\begin{equation}
    S^{[l]} = \{s_1\cdot {\dots}\cdot s_n\mid s_i\in S,\ 0\leq n\leq l\}.
\end{equation}
Note that for $n=0$ we obtain $1\in S^{[l]}.$

\begin{proposition}\label{prop:l-covering-Caylay}
Let $\varphi :H\to G$ be a group epimorphism and let $T\subseteq H\setminus\{1\},$ $S\subseteq G\setminus\{1\}$ be generating subsets such that $\varphi$ induces bijections $T\cong S$ and $T^{[l]}\cong S^{[l]}.$ Then the morphism of digraphs
\begin{equation}
\Cay(H,T) \to \Cay(G,S)
\end{equation}
is an $l$-covering. 
\end{proposition}
\begin{proof}
Since the map $T\to S$ is a bijection, it is obvious that the map is a $1$-covering. By Proposition \ref{prop:covering_description1}, in order to check that it is an $l$-covering, it is sufficient to check that for any paths $(h_0,\dots,h_n)$ and $(h'_0,\dots,h'_m)$ in $\Cay(H,T)$ such that $h_0=h'_0, \varphi(h_n)=\varphi(h'_m)$ (resp. $\varphi(h_0)=\varphi(h'_0), h_n=h'_m$) and $0\leq n,m\leq l$ we have $h_0=h'_0, h_n=h'_m.$ Assume that  $h_0=h'_0, \varphi(h_n)=\varphi(h'_m).$  Consider the sequences $t_1,\dots,t_n$ and $t_1',\dots,t'_m$ from $T\cup \{1\}$ such that $h_i=h_0 \tilde t_1\dots \tilde t_i$ and $h'_j=h_0\tilde t'_1\dots \tilde t'_j.$ The equation $\varphi(h_n)=\varphi(h'_m)$ implies the equation $\varphi(t_1\dots t_n ) = \varphi( t'_1\dots t'_m)$ in $G.$ By the assumption, we obtain $t_1\dots t_n=t'_1\dots t'_m.$ Therefore $ h_n=h'_m.$ 
\end{proof}

We say that the generating subset $S\subseteq G$ is an \emph{$l$-freely generating set} of $G,$ if $\Rel(G,S)=\Rel^l(G,S).$ 
Proposition \ref{prop:pi^l-Cay} implies that $S$ is $l$-freely generating set if and only if $\pi_1^l({\sf Cay}(G,S))=1.$ 

\begin{example} \ 
\begin{enumerate}
\item $S$ is $1$-freely generating if and only if the map $F(S)\to G$ is an isomorphism.
\item $G$ is $2$-freely generating set for itself, because $\Rel(G,G)$ is generated by elements  $\bar g \cdot \bar h \cdot  (\overline{gh})^{-1}\in W^2(G).$ 
\item If $G=\ZZ^n$ then the standard basis $S=\{e_1,\dots,e_n\}$ is a $2$-freely generating set of $G.$

\item If $G=\ZZ/n\ZZ,$ then $S=\{1+n\ZZ\}$ is $n$-freely generating set for $G$ but not $(n-1)$-freely generating.
\end{enumerate}
\end{example}

We set 
\begin{equation}
F^l(G,S) := F(S)/\Rel^l(G,S).
\end{equation}
and denote by $\tilde S$ the image of $S$ in $ F^l(G,S),$ and by $\tilde s\in \tilde S$ the image of an element $s\in S.$ It is easy to see that $\Rel^l(G,S)=\Rel^l(F^l(G,S),\tilde S).$ Therefore $\tilde S$ is $l$-freely generating in $F^l(G,S).$ The homomorphism $\sigma$ induces an epimorphism $\sigma^l:F^l(G,S)\epi G,$ which induces bijections $\tilde S\cong S$ and $\tilde S^{[l]}\cong S^{[l]}.$ Note that $S$ is $l$-freely generating if and only if the map $\sigma^l:F^l(G,S)\to G$ is an isomorphism. Moreover $\sigma^l$ is a universal homomorphism from a group with an $l$-generating subset in the following sense. 

\begin{proposition}\label{prop:universal-F^l}
Let $\varphi : H\to G$ be a group epimorphism and let $T\subseteq H,$ $S\subseteq G$ be generating subsets such that $\varphi$ induces bijections $T\cong S$ and $T^{[l]}\cong S^{[l]}.$ Assume that $T$ is $l$-freely generating.  Then there is a unique isomorphism 
\begin{equation}
\theta:F^l(G,S) \overset{\cong}\longrightarrow H 
\end{equation}
such that $\varphi(\tilde S)=T$ and the diagram 
\begin{equation}
\begin{tikzcd}
F^l(G,S)\ar[rr,"\cong"] \ar[rd,"\sigma^l"'] & & H \ar[ld,"\varphi"] \\
& G &    
\end{tikzcd}
\end{equation}
is commutative. 
\end{proposition}
\begin{proof} If such an isomorphism exists, the commutativity of the diagram implies that $\theta(\tilde s)=\varphi^{-1}(s).$ This proves that, if it exists, it is unique.  
Consider a homomorphism $\theta':F(S) \to H$ defined by $\theta'(\bar s)=\varphi^{-1}(s).$ Let us check that $\Rel^l(G,S)\subseteq \Ker(\theta').$ Take $(\bar s_1 \dots \bar s_n)\cdot (\bar s'_1\dots \bar s'_m)^{-1}\in W^l(G,S)\cap \Rel(G,S),$ where $n,m\leq l.$  Then $\varphi(\theta'(\bar s_1)\dots \theta'(\bar s_n) )=s_1\dots s_n=s'_1\dots s'_m = \varphi(\theta'(\bar s'_1)\dots \theta'(\bar s'_m))$ in $G.$ Since $T^{[l]}\to S^{[l]}$ is a bijection, we obtain $\theta'(\bar s_1)\dots \theta'(\bar s_n) = \theta'(\bar s'_1)\dots \theta'(\bar s'_m).$ Therefore $(\bar s_1 \dots \bar s_n)\cdot (\bar s'_1\dots \bar s'_m)^{-1}\in \Ker(\theta').$ Hence, $\theta'$ induces a homomorphism $\theta:F^l(G,S)\to H$ such that $\theta(\tilde s)=\varphi^{-1}(s).$  Let us prove that $\theta$ is an isomorphism. The homomorphism $\varphi$ induces a homomorphism $F^l(\varphi):F^l(H,T)\to F^l(G,S).$ Since $T$ is $l$-freely generating, $\sigma^l:F^l(H,T) \to H$ is an isomorphism. This defines a homomorphism $ F^l(\varphi)(\sigma^l)^{-1}: H\to F^l(G,S).$ It is easy to check that this homomorphism defines the inverse bijection on generators $T\overset{\cong}\to \tilde S.$ Therefore $\theta$ is an isomorphism. 
\end{proof}

\begin{proposition}\label{prop:cay_cover}
The morphism of Cayley digraphs 
\[{\sf Cay}(F^l(G,S),\tilde S) \to {\sf Cay}(G,S)\] is the universal $l$-covering. 
\end{proposition}
\begin{proof}
Since $\tilde S$ is $l$-freely generating in $F^l(G,S)$, by Proposition \ref{prop:pi^l-Cay} the group $\pi_1^l( {\sf Cay}(F^l(G,S),\tilde S) )$ is trivial. Therefore, we just need to prove that the map is an $l$-covering. By the definition of $\Rel^l(G,S)$ it is easy to see that $\tilde S^{[l]} \to  S^{[l]}$ is a bijection. Then the fact that it is an $l$-covering follows from Proposition \ref{prop:l-covering-Caylay}.
\end{proof}

\subsection{Path homology of a Cayley digraph}

For a group $G$ and a generating subset $S\subseteq G$ we set
\begin{equation}
\tilde G = F^2(G,S).
\end{equation}

\begin{theorem}\label{theorem:main}
Let $G$ be a group and $S$ be its generating subset and $\rho:=\rho^2(G,S).$ Then there is a spectral sequence of homological type
\begin{equation}
E^2_{i,j} = H_i( \rho, \PH_j( \Cay(\tilde G,\tilde S) )  ) \Rightarrow \PH_{i+j}(\Cay(G,S)),
\end{equation}
where $H_*(\rho,-)$ is the group homology of the group $\rho,$ and $\PH_*( \Cay(\tilde G,\tilde S) )$ is treated as a $\KK[\rho]$-module, where the action of $\rho\subseteq \tilde G$ is defined by shifts in $\tilde G.$ 
\end{theorem}
\begin{proof}
It follows from Theorem \ref{theorem:covering_spectral} and Proposition \ref{prop:cay_cover}. 
\end{proof}

We say that a digraph $X$ is $\KK$-acyclic, if $\tilde \PH_*(X)=0.$

\begin{corollary}\label{cor:main}
Under assumptions of Theorem \ref{theorem:main}, if  $\Cay(\tilde G,\tilde S)$ is $\KK$-acyclic, then 
\begin{equation}
\PH_*( \Cay(G,S) ) \cong  H_*( \rho ). 
\end{equation}
\end{corollary}

\subsection{Cayley digraphs of abelian groups}

In this subsection we will consider abelian groups with additive notation. We say that a digraph $X$ is $\KK$-acyclic, if $\tilde \PH_*(X)=0.$ 

\begin{lemma}\label{lemma:Z1} 
The digraph $\Cay(\ZZ,\{1\})$ is $\KK$-acyclic for any $\KK$.
\end{lemma}
\begin{proof}
The digraph $\Cay(\ZZ,\{1\})$ is the union of finite induced line digraphs $I^{-n,n}$ on the set $\{-n,-n+1,\dots,n-1,n\}$ for $n\geq 1.$ Then $\Cay(\ZZ,\{1\})={\sf colim} I^{-n,n}.$ By Proposition \ref{prop:filtered_colimits} we obtain $\PH_*(\Cay(\ZZ,\{1\})) = {\sf colim}\: \PH_*(I^{-n,n}).$ Then the assertion follows from the fact that $I^{-n,n}$ is contractible. 
\end{proof}

\begin{lemma}\label{lemma:ref}
Let $G$ be a free abelian group with a basis $S.$ Then $\Cay(G,S)$ is $\KK$-acyclic for any $\KK.$    
\end{lemma} 
\begin{proof}
If $S$ is finite, then it follows from Lemma \ref{lemma:Z1}, the formula 
\begin{equation}
\Cay(\ZZ^n,\{s_1,\dots,s_n\}) = \overset{n}{\underset{i=1}\Box} \Cay(\ZZ,\{1\}),
\end{equation}
where $s_1,\dots,s_n$ is the standard basis, and the K\"unneth theorem for the box product. If $S$ is infinite, we use the fact that $F$ is the filtered colimit of finitely generated free subgroups $\langle S' \rangle,$ where $S'\subseteq S$ is finite, and use Proposition \ref{prop:filtered_colimits}.
\end{proof}

\begin{theorem}\label{th:abelian}
Let $G$ be an abelian group and $S\subseteq G\setminus\{0\}$ be a generating subset such that
\begin{enumerate}
    \item for any $s,s'\in S$ we have $s+s'\notin S\cup \{0\};$
    \item for any $s_1,s_1',s_2,s_2'\in S$ the equation $s_1+s_1'=s_2+s_2'$ implies that either $(s_1,s_1')=(s_2,s_2')$ or $(s_1,s_1')=(s_2',s_2).$ 
\end{enumerate}
Consider the canonical epimorphism from the free abelian group $\ZZ^{\oplus S}\epi G$ and denote its kernel by 
$\rho = {\sf Ker}(\ZZ^{\oplus S}\epi G).$
Then there is an isomorphism
\begin{equation}
\PH_*(\Cay(G,S)) = \Lambda^* (\rho \otimes_\ZZ \KK), 
\end{equation}
where $\Lambda^*$ denotes the exterior algebra over $\KK.$  
\end{theorem}
\begin{proof} It is easy to see that the standard basis $\tilde S=\{e_s\mid s\in S\}$ in $\ZZ^{\oplus S}$ is a 2-freely generating subset and the map $\tilde S \to S$ is a bijection. The assumptions (1),(2) imply that $\tilde S^{[2]}\to S^{[2]}$ is a bijection. Then by Proposition \ref{prop:universal-F^l} we obtain 
$\tilde G=\ZZ^{\oplus S}.$
Lemma \ref{lemma:ref} implies that $\Cay(\tilde G,\tilde S)$ is $\KK$-acyclic. Then Corollary \ref{cor:main} implies that $\PH_*(\Cay(G,S))=H_*(\rho).$ Since $\rho$ is a subgroup of a free abelian group, $\rho$ is also a free abelian, and hence $H_*(\rho)=\Lambda^* (\rho \otimes \KK).$ 
\end{proof}

\begin{corollary}
Under the assumption of Theorem \ref{th:abelian}, the modules $\PH_n(\Cay(G,S))$ are free over $\KK.$
\end{corollary}

On the next proposition we give an example of usage of Theorem \ref{th:abelian}. 

\begin{proposition} There is an isomorphism 
\[
\PH_*(\Cay( \QQ,\left\{1/n!\mid n\geq 2 \right\}))\cong \Lambda^*  (\KK^{\bigoplus \mathbb{N}}).   
\]
\end{proposition}
\begin{proof}
First we need to check the assumptions of Theorem \ref{th:abelian}. 

(1) Assume the contrary, that $\frac{1}{n!} + \frac{1}{m!}=\frac{1}{k!}$ for $n,m,k\geq 2.$ Then $n,m>k.$ Without loss of generality we can assume that $n\geq m>k\geq 2.$ Multiplying by $n!,$ we obtain an equation for integer numbers
\begin{equation}\label{eq:fact1}
1+\frac{n!}{m!}=\frac{n!}{k!}.    
\end{equation}
If $n>m,$ then $\frac{n!}{m!}$ and $\frac{n!}{k!}$ are divisible by $n,$ and this contradicts to the equation \eqref{eq:fact1}. If $n=m,$ then $\frac{n!}{k!}$ is divisible by $n,$ which contradicts to the equation $2=\frac{n!}{k!}.$ 

(2) Assume that $\frac{1}{n!}+\frac{1}{m!}=\frac{1}{k!}+\frac{1}{l!}.$ Without loss of generality we can assume that $n\geq m,k,l.$ Multiplying by $n!$ we obtain an equation for integers
\begin{equation}\label{eq:fact2}
    1+\frac{n!}{m!} = \frac{n!}{k!} + \frac{n!}{l!}. 
\end{equation} 
If $n>m,l,k,$ then $\frac{n!}{m!}, \frac{n!}{k!}, \frac{n!}{l!}$ are divisible by $n,$ which contradicts to the equation \eqref{eq:fact2}. Therefore $n\in \{m,k,l\}.$ If $n=m>k,l$ then $\frac{1}{n!} + \frac{1}{m!}< \frac{1}{k!} + \frac{1}{l!}.$ Therefore $n\in \{k,l\}.$ If $n=k,$ then it is easy to see that $m=l;$ and dually, if $n=l$, then $k=m.$

We proved the assumption of Theorem \ref{th:abelian}. Therefore, we just need to prove that $\rho\cong \ZZ^{\oplus \mathbb{N}}.$ By the definition $\rho$ is the kernel of the map $\ZZ^{\oplus \mathbb{N}}\to \QQ$ that sends the element of the standard basis $e_n$ to $\frac{1}{(n+1)!}.$ So $\rho$ is a subgroup of a free abelian group, and then, $\rho$ is also free abelian. So we just need to check that $\rho$ is not finitely generated. Indeed $\{(n+2)e_{n+1}-e_{n}\mid n\in \mathbb{N} \}$ is an infinite linearly independent subset in $\rho,$ and thus, $\rho\cong \ZZ^{\oplus \mathbb{N}}.$
\end{proof}

Further we give some examples that do not satisfy the assumption of Theorem \ref{th:abelian}. 

\begin{proposition}\label{prop:Z2^n}
The digraph $\Cay(\ZZ, \{2^i\mid 0\leq i\leq n\})$ is $\KK$-acyclic for any $\KK$ and $n\geq 0.$   
\end{proposition}
\begin{proof} Set $X^n=\Cay(\ZZ, \{2^i\mid 0\leq i\leq n\}).$
The proof is by induction on $n.$ For $n=0$ it is the statement of Lemma \ref{lemma:Z1}. Assume that the statement holds for $n,$ and prove it for $n+1.$ Consider two maps $f,g:\ZZ\to \ZZ,$ where $f(x)=2x$ and $g(x)=\lfloor x/2 \rfloor.$ It is easy to see that $f$ defines a morphism of digraphs $f:X^n\to X^{n+1}.$ We also claim that $g$ defines a morphism of digraphs $g:X^{n+1}\to X^n.$ Let us check it. There are arrows of four types in $X^{n+1}:$ $(2x,2x+1),$ $(2x+1,2x+2),$ $(2x,2x+2^i),$ $(2x+1,2x+1+2^i)$ where $1\leq i\leq n+1.$ If we apply $g,$ we obtain the following pairs, which are either arrows or equations: $(x,x),$ $(x,x+1),$ $(x,x+2^{i-1}),$ $(x,x+2^{i-1}).$ The composition $gf:X^n\to X^{n}$ is identical. The composition $fg:X_{n+1}\to X^{n+1}$ is defined by the formula
\begin{equation}
(fg)(x)=\begin{cases}
    x, & \text{ if } x \text{ is odd}\\
    x-1, & \text{ if } x \text{ is even}.
\end{cases}
\end{equation}
Hence, for any $x\in \ZZ$ the pair $(fg(x),x)$ is in $X^{n+1}_1.$ Therefore, $fg$ is homotopic to ${\sf id}.$ It follows that $\PH_*(X^{n+1})=\PH_*(X^n).$
\end{proof}

\begin{proposition}\label{prop:Z[1/2]}
Let $\ZZ[\frac{1}{2}]$ denote the group of rational numbers of the form $m/2^n.$ Then the digraph $\Cay( \ZZ[\frac{1}{2}], \{\frac{1}{2^i}\mid i\geq 0\} )$ is $\KK$-acyclic for any $\KK.$  
\end{proposition}
\begin{proof}
It is easy to see that $\Cay( \ZZ[\frac{1}{2}], \{\frac{1}{2^i}\mid i\geq 0\} ) = \bigcup_{n\geq 0}  \Cay( \frac{1}{2^n} \ZZ, \{\frac{1}{2^i}\mid 0\leq i\leq n\} ).$ By Proposition \ref{cor:colimit}, it is sufficient to prove that $\Cay( \frac{1}{2^n} \ZZ, \{\frac{1}{2^i}\mid 0\leq i\leq n\} )$ is $\KK$-acyclic. This follows from the fact that the multiplication on $2^n$ defines an isomorphism between digraphs $\Cay( \frac{1}{2^n} \ZZ, \{\frac{1}{2^i}\mid 0\leq i\leq n\} )$ and $ \Cay( \ZZ, \{2^i \mid 0\leq i\leq n\} ),$ and Proposition \ref{prop:Z2^n}.
\end{proof}

\begin{proposition}
Consider the circle group 
$\{z\in \mathbb{C} \mid |z|=1 \}$
and its subgroup $G=\{e^{m \pi/2^n} \mid n\geq 1\}.$ Then 
\begin{equation}
  \PH_*(\Cay(G,\{e^{\pi/2^n}\mid n\geq 1\})) = H_*(S^1),   
\end{equation}
where $H_*(S^1)$ is the homology of the circle. 
\end{proposition}
\begin{proof}
Note that $G\cong \ZZ[ \frac{1}{2}]/\ZZ$ and this isomorphism takes $\{e^{\pi/2^n}\mid n\geq 1\}$ to 
$S:=\{\frac{1}{2^{n+1}} +\ZZ \mid n\geq 1 \}.$ The set  $\tilde S=\{\frac{1}{2^{n+1}} \mid n\geq 1  \} $ is $2$-freely generating in $\ZZ[\frac{1}{2}]$ and the maps $\tilde S\to S $ and $\tilde S^{[2]}\to S^{[2]}$ are bijections (here it is important that $\frac{1}{2^{n+1}}+\frac{1}{2^{m+1}}\notin \ZZ$ for $n,m\geq 1$ because $0\in \tilde S^{[2]}$). By Proposition \ref{prop:universal-F^l} we obtain $\tilde G=\ZZ[\frac{1}{2}].$ Therefore $\rho=\ZZ.$ Note that $H_*(\ZZ)=H_*(S^1).$ Then the assertion follows from 
Corollary \ref{cor:main} and  Proposition \ref{prop:Z[1/2]}.
\end{proof}

\printbibliography

\end{document}